\documentclass{article}
\usepackage[utf8]{inputenc}
\usepackage{amsthm}
\usepackage{amsmath}
\usepackage{amsfonts}
\usepackage{color}
\usepackage{latexsym}
\usepackage{geometry}
\usepackage{enumerate}
\usepackage[shortlabels]{enumitem}
\usepackage{csquotes}
\usepackage{graphicx}
\usepackage{comment}
\usepackage{appendix}
\usepackage{hyperref}
\hypersetup{
    colorlinks=true,
    linkcolor=blue,
    filecolor=magenta,      
    urlcolor=cyan,
}
\usepackage{bm}
\usepackage{amssymb}

\emergencystretch=\maxdimen
\def \dist {{\rm dist}}
\def \IR {\mathbb{R}}

\DeclareMathOperator{\Rm}{Rm}
\DeclareMathOperator{\Ric}{Ric}

\DeclareMathOperator{\spt}{spt}
\makeatletter
\newcommand*{\rom}[1]{\rm {\expandafter\@slowromancap\romannumeral #1@}}
\makeatother

\def\XXint#1#2#3{{\setbox0=\hbox{$#1{#2#3}{\int}$ }
\vcenter{\hbox{$#2#3$ }}\kern-.6\wd0}}

\protected\def\vts{%
  \ifmmode
    \mskip0.5\thinmuskip
  \else
    \ifhmode
      \kern0.08334em
    \fi
  \fi
}

\numberwithin{equation}{section}
\newtheorem{Theorem}{Theorem}[section]
\newtheorem{Proposition}[Theorem]{Proposition}
\newtheorem{Lemma}[Theorem]{Lemma}
\newtheorem{Corollary}[Theorem]{Corollary}
\theoremstyle{definition}

\def \W {\overline{W}}

\title{On noncollapsed $\mathbb{F}$-limit metric solitons}
\author{Pak-Yeung Chan\footnote{Pak-Yeung Chan's research is supported by EPSRC grant EP/T019824/1.}, Zilu Ma, Yongjia Zhang\footnote{Yongjia Zhang's research is  supported by National Natural Science Foundation of China NSFC12301076 and Shanghai Sailing Program 23YF1420400.}}
\date{}

\begin{document}


\maketitle

\begin{abstract}
   A noncollapsed $\mathbb{F}$-limit metric soliton is a self-similar singularity model that inevitably arises when studying the Ricci flow with the tool of $\mathbb{F}$-convergence \cite{Bam20a,Bam20b,Bam20c}. In this article, we shall present a systematic study of the noncollapsed $\mathbb{F}$-limit metric soliton, and show that, apart from the known results in \cite{Bam20c}, it satisfies many properties of smooth Ricci shrinkers. In particular, we show a quadratic lower bound for the scalar curvature, a local gap theorem, a global Sobolev inequality, and an optimal volume growth lower bound.
\end{abstract}


\tableofcontents

\section{Introduction}

In the development of the Ricci flow, especially in the proof of the geometrization conjecture via the Hamilton-Perelman program \cite{Ham93,Per02,Per03a,Per03b}, singularity analysis is a critical problem. The efficiency of the singularity analysis, to certain degree, determines the potential of the Ricci flow in its geometric application.  In Hamilton's classical singularity analysis \cite{Ham93}, strong local a priori estimates (especially the ``bounded curvature at bounded distance condition'') are essential to the existence of singularity models. In other words, one cannot freely detect the local geometry at arbitrary regions; only the geometric information at the highest-curvature region is available. Although Perelman improved Hamilton's method and overcame this difficulty in the three-dimensional case, the singularity analysis in general higher dimensional cases remains a major issue in the field of Ricci flow.

Bamler's breakthrough \cite{Bam20a,Bam20b,Bam20c} provides some remedy to the aforementioned problem. By developing sophisticated estimates for the heat kernel and the entropy on the Ricci flow \cite{Bam20a}, Bamler invented proper objects for compactness theory, called metric flows \cite{Bam20b}. A metric flow is a family of time-dependent metric spaces, with evolving probability measures that behave like conjugate heat kernels. In particular, a Ricci flow can be regarded as a metric flow. The convergence of metric flows is called the $\mathbb{F}$-convergence. Just like the pointed-Gromov-Hausdorff convergences, one needs also to pick a sequence of  ``base points'' for the $\mathbb{F}$-convergence, which is, in reality, a sequence of conjugate heat kernels.

\emph{Any} sequence of $n$-dimensional Ricci flows convergences, possibly after passing to a subsequence, to a metric flow. This compactness result, if applied to the singularity analysis, naturally leads to a more general notion of singularity model---the $\mathbb{F}$-limit of smooth Ricci flows. Bamler \cite{Bam20c} showed many geometric properties for an $\mathbb{F}$-limit of a noncollapsed sequence of $n$-dimensional Ricci flows; here by noncollapsed we mean that the sequence of Ricci flows has bounded Nash entropy at some fixed time. Among many others, the most important properties of such an $\mathbb{F}$-limit are: (1) it is almost everywhere smooth, and the singular set has space-time codimension at least four; (2) the convergence is a smooth convergence on the regular part of the limit, and consequently, the regular part of the $\mathbb{F}$-limit is a smooth Ricci flow space-time.

If, in addition to the noncollapsing assumption, the $\mathbb{F}$-limit has constant entropy, then it is a self-similar metric flow called a metric soliton (see Section 2 for more details). Metric solitons are not only special $\mathbb{F}$-limits, but arise in many scenarios, and therefore are important objects of study. In fact, so long as the monotonicity of the Nash entropy comes into play, one expects the $\mathbb{F}$-limit to be a metric soliton. For instance, the tangent flow at one point, or the tangent flow of an ancient solution at infinity, are both metric solitons.

Not surprisingly, the metric soliton satisfies the same differential equation as the smooth Ricci shrinker on the regular part. Furthermore, it is known that a four-dimensional $\mathbb{F}$-limit metric soliton has only isolated orbifold singularities. But little is known beyond these facts. For classical smooth Ricci shrinkers, many nice estimates and analyses are obtained, among which quite a few promoted the study of the Ricci flow. One naturally wonders whether the same or similar results are true for metric solitons, at least on its regular part. The goal of the current paper is to show some of these results. Hopefully our work will be helpful to the singularity analysis in higher dimensions. In order to present our project, we shall review some classical results of the Ricci shrinker, and our plan will flush out naturally.

Let $(M^n,g,f)$ be a complete and smooth Ricci shrinker, satisfying the following equation and normalization conditions.
\begin{gather*}
\Ric+\nabla^2f=\frac{1}{2}g,
\\
\int_M(4\pi)^{-\frac{n}{2}}e^{-f}\,dg=1,
\\
\mu_g:\equiv 2\Delta f-|\nabla f|^2+R+f-n<0.
\end{gather*}
We summarize some classical results below.
\begin{enumerate}[(1)]
\item (\cite{CZ10}) The potential function has quadratic growth:
$$\frac{1}{4}\left(d(x,x_0)-5n\right)^2_+\le f(x)-\mu_g\le \frac{1}{4}\left(d(x,x_0)+\sqrt{2n}\right)^2,$$
where $x_0$ is a minimum point of $f$.

\item (\cite{CLY11}) A quadratic lower bound of the scalar curvature (so long as the shrinker is nonflat):
$$R(x)\ge \frac{c}{f(x)-\mu_g}\qquad\text{ for all }\qquad x\in M.$$
Here $c$ is a constant depending on the shrinker.

\item (\cite{CMZ22}) A local gap theorem: if there is a large region in the shrinker with almost-Euclidean $\mu$-functional, namely
$$\mu\left(B(x_0,\delta^{-1}),g,1\right)\ge -\delta,$$
where $\delta=\delta(n)$ and $x_0$ is a minimum point of $f$, then the shrinker is the Euclidean space.

\item (\cite{LW20}) A global Sobolev inequality, or equivalently, Perelman's $\nu$-functional is finite on a shrinker:
$$\nu(g)=\mu_g.$$ 

\item (\cite{MW12,LW20}) An optimal volume growth lower bound: 
$$|B(x_0,r)|\ge c(n)e^{\mu_g}r\qquad\text{ for all }\qquad r\ge C(n).$$

\end{enumerate}

In the rest part of the current section, we shall introduce our main theorems. It is therefore imperative to briefly introduce some basic notions which appear in our statements. Our work is built upon Bamler's theory \cite{Bam20a,Bam20b,Bam20c}, which is a massive system containing quite a few definitions and results. It is impossible to introduce every theorem and notion applied in this paper. So we assume of the reader some basic familiarity with Bamler's work, and will refer to \cite{Bam20a,Bam20b,Bam20c} whenever necessary.

Throughout this article, we consider an \emph{$n$-dimensional noncollapsed $\mathbb{F}$-limit metric soliton} $\left(\mathcal{X},(\nu_t)_{t\in(-T,0)}\right)$ over the time-interval $(-T,0)$. The terminology here implies that such a metric soliton arises as an $\mathbb{F}$-limit of a sequence of $n$-dimensional Ricci flows with bounded Nash entropy. Since a metric soliton is self-similar, it is modelled on some measured metric space $(X,d,\nu)$, which is called its \emph{model}. According to \cite{Bam20c}, the model space is a metric completion of a smooth Remannian manifold, called its \emph{regular part}. We shall denote by $(\mathcal{R}_X,g,f_0)$ the regular part of the model $(X,d,\nu)$. Note that $$d\nu=(4\pi)^{-\frac{n}{2}}e^{-f_0}dg \quad \text{ on }\quad \mathcal{R}_X.$$ The smooth function $f_0:\mathcal{R}_X\to\mathbb{R}$ is also called \emph{the potential function} as in the smooth case. Another notion needed in the statements below is called \emph{center of the metric soliton}. It is a point $x_0\in\mathcal{R}_X$ in the regular part which is almost equivalent to the minimum point of the potential function in the smooth case. Henceforth, the notation $x_0$ will be reserved for a center of a metric soliton. Finally, we shall call $$ W:=\int_{\mathcal{R}_X}f_0 (4\pi)^{-\frac{n}{2}}e^{-f_0}\,dg-\frac{n}{2}$$
the \emph{soliton entropy}. It is in fact the Nash entropy of the metric soliton defined using the conjugate heat flow $\nu_t$ (which is a constant), and is equivalent to the $\mu_g$ constant of a smooth shrinker. A more thorough treatment of these notions is included in Section 2 and Section 3.

\subsection{Basic properties of the potential function and the volume}

To begin with, the potential function $f_0$ of a metric soliton also has quadratic growth rate. 

\begin{Theorem}[Potential function estimate]\label{thm:potential-function}
Let $(X,d,\nu)$ be the model of an $n$-dimensional noncollapsed $\mathbb{F}$-limit metric soliton and $(\mathcal{R}_X,g,f_0)$ the regular part of the model. Let $x_0\in\mathcal{R}_X$ be a center of the soliton. Then, for any $\varepsilon>0$, we have
\begin{align*}
    \frac{1}{8+\varepsilon}d^2_g(x_0,x)-C(\varepsilon)\le f_0(x)-W\le \frac{1}{4}\big(d_g(x_0,x)+C(n)\big)^2,
\end{align*}
where $C(\varepsilon)$ is a constant depending only on $\varepsilon$.
\end{Theorem}

It is conceivable that the upper bound follows from integrating some classical soliton identity (cf. the third equation in \eqref{eq: basic-properties-2}) along almost-minimizing curves contained in $\mathcal{R}_X$. So the key of proving the upper bound is to show that $f_0(x_0)-W$ is bounded by a dimensional constant. The lower bound is in fact a consequence of Bamler's conjugate heat kernel estimate \cite[Theorem 7.2]{Bam20a} and the local smooth convergence property of the $\mathbb{F}$-convergence \cite[Theorem 2.5]{Bam20c}. It goes without saying that this idea can be applied to the volume. Therefore the following two corollaries are immediately at hand.

\begin{Corollary}[Volume ratio upper bound at a center]\label{volume-upper-1}
    In the same settings as Theorem \ref{thm:potential-function}, we have
    \begin{align*}
        \left|\mathcal{R}_X\cap B(x_0,r)\right|\le C(n)e^W r^n \quad \text{ for all }\quad r\ge 1.
    \end{align*}
\end{Corollary}

\begin{Corollary}[Volume ratio upper bound]\label{volume-upper-2}
    In the same settings as Theorem \ref{thm:potential-function}, we have
    \begin{align*}
        \left|\mathcal{R}_X\cap B(x,r)\right|\le C(n) r^n \quad \text{ for all }\quad r\ge 1,
    \end{align*}
where $x$ is any point in $X$.
\end{Corollary}

The proofs of the results in this subsection can be found in Section 4.

\subsection{A lower bound for the scalar curvature}

A result of Chow-Lu-Yang \cite{CLY11} shows that the scalar curvature of any nonflat Ricci shrinker has at most quadratic decay. This estimate is sharp due to the existence of asymptotically conical Ricci shrinkers \cite{FIK03}. The Chow-Lu-Yang estimate is generalized to ancient solutions by Chow and the authors in \cite{CCMZ23}, with an additional assumption that the scalar curvature around the $H_n$-centres of a fixed point is uniformly bounded from below. The theorem below shows that the Chow-Lu-Yang estimate also holds on the regular part of the model of a metric soliton, unless it is Ricci flat.

\begin{Theorem}[Scalar curvature lower bound]\label{thm: chow-lu-yang estimate}
    Let $(X,d,\nu)$ be the model of an $n$-dimensional noncollapsed $\mathbb{F}$-limit metric soliton and $(\mathcal{R}_X,g,f_0)$ its regular part. Assume furthermore that $g$ is non-Ricci-flat. Let $W$ be the soliton entropy and define 
    \begin{align}\label{eq:scalar-integral-positivity}
        c_R:=&\ \int_{\mathcal{R}_X} R\,(4\pi)^{-\frac{n}{2}}e^{-f_0}\,dg>0
    \end{align}
    Then there is a constant $c=c(n,W,c_R)$ depending only on $W$ and $c_R$, such that
    \begin{align*}
        R(x)\ge \frac{c}{f_0(x)-W+C_0}\quad\text{ for all  }\quad x\in M,
    \end{align*}
    where $C_0$ is a dimensional constant.
\end{Theorem}

In combination with Theorem \ref{thm:potential-function}, the lower bound for the scalar curvature can obviously be rewritten as
$$R(x)\ge \frac{c(n,W,c_R)}{d_g^2(x_0,x)+C_0(n)}\qquad \text{ for all }\qquad x\in\mathcal{R}_X,$$
where $x_0\in\mathcal{R}_X$ is a center of the metric soliton.

Recall that the method of \cite{CLY11} is to apply the maximum principle to an auxiliary function constructed with the scalar curvature and the potential function, while in \cite{CCMZ23}, the auxiliary function implemented is a parabolic version constructed with the scalar curvature and the conjugate heat kernel. The proof of Theorem \ref{thm: chow-lu-yang estimate} is more similar to (and more complicated than) the method of \cite{CCMZ23}. In fact, we apply the non-Ricci-flat assumption to obtain a local lower bound for the scalar curvatures along the $H_n$-centers of the limiting sequence. And we argue that the sequence satisfies an almost quadratic scalar curvature lower bound by applying the method in \cite{CCMZ23}. Finally, the local smooth convergence property \cite[Theorem 2.5]{Bam20c} shows that the lower bound of the scalar curvature can be carried to the regular part of the metric soliton; the proof of Theorem \ref{thm: chow-lu-yang estimate} is presented in Section 5.

\subsection{A local gap theorem}

In \cite{CMZ22}, the authors proved a local gap theorem for Ricci shrinkers: if the local $\mu$-functional of a large domain in the shrinker is almost zero, then the shrinker must be an Euclidean space. We shall prove the same result for metric solitons. Note that the region of the local $\mu$-functional must be contained in the regular part, for otherwise the local $\mu$-functional is not well-defined.

\begin{Theorem}[Local gap theorem]\label{thm: local-gap}
    There is a positive number $\delta=\delta(n)>0$ depending only on the dimension $n$, with the following property. Let $(X,d,\nu)$ be the model of an $n$-dimensional noncollapsed $\mathbb{F}$-limit metric soliton and $(\mathcal{R}_X,g,f_0)$ its regular part. Let $x_0\in\mathcal{R}_X$ be a center of the soliton. If
    \begin{align*}
        \mu\Big(B(x_0,\delta^{-1})\cap\mathcal{R}_X,g,1\Big)\ge -\delta,
    \end{align*}
    where $\mu$ is Perelman's local $\mu$-functional defined in Subsection 6.1, then the metric soliton is the Gaussian soliton modelled on the Euclidean space.
\end{Theorem}


In \cite{CMZ22}, the authors presented three different proofs for the local gap theorem of smooth shrinkers. The most easy but least intuitive one is to modify the function $(4\pi)^{-\frac{n}{2}}e^{-f}$ by a cut-off method, and implement it as the test function of the local $\mu$-functional. Then, the smallness of the local $\mu$-functional implies the smallness of the soliton entropy, and the flatness follows from the Carrillo-Ni-Yokota gap theorem \cite{CN09,Y09,Y12}. The other two approaches are more intuitive. One may apply the pseudolocality theorem of \cite{LW20} and show that under the smallness of the local $\mu$-function, the scalar curvature remains bounded until the singular time, which cannot happen to nonflat shrinkers. One may also apply the theory of \cite{CMZ23a}  to see that the shrinker in question must have small Nash entropy everywhere at arbitrary large scales, and the conclusion follows from Bamler's $\varepsilon$-regularity theorem \cite[Theorem 10.2]{Bam20a}.

Our approach to Theorem \ref{thm: local-gap} is more close to the first among the three mentioned above. Although we do not have a Carrillo-Ni-Yokota gap theorem for the soliton entropy, yet the soliton entropy, being the same as the Nash entropy, is equal to the limit of the Nash entropy of the convergence sequence. Thus, we can show a preliminary soliton entropy gap as presented in  Theorem \ref{thm-Nash-gap}. Next, to apply $(4\pi)^{-\frac{n}{4}}e^{-f_0}$ as the test function of the local $\mu$-functional, one needs to modify it, such that it is compactly supported not only on a compact ball, but also on the regular part. Thus, given a proper cut-off function, the rest of the proof follows in the same way as \cite{CMZ22}. We shall modify Bamler's cut-off function \cite[Lemma 15.27]{Bam20c} for metric solitons. Due to its importance for proving other main theorems, we shall include it here.

\begin{Proposition}\label{prop:regular cutoff}
   Let $(X,d,\nu)$ be the model of an $n$-dimensional noncollapsed $\mathbb{F}$-limit metric soliton and $(\mathcal{R}_X,g,f_0)$ its regular part. Let $x_0\in\mathcal{R}_X$ be a center of the soliton and $W$ the soliton entropy. If $$\sigma>0,\quad A>\underline{A}(n),$$
   then for any $r\le \overline{r}(A,W,\sigma)$  there is a continuous function $\eta_r: X\to [0,1]$ satisfying the following properties.
   \begin{enumerate}[(1)]
       \item $\eta_r$ is supported on $\mathcal{R}_X$. Namely, $\eta_r\equiv 0$ on an open neighborhood of $X\setminus\mathcal{R}_X$.
       \item $\eta_r$ is smooth on $\mathcal{R}_X$, satisfying $|\nabla_g\eta_r|\le C_0r^{-1}$, where $C_0$ is a dimensional constant.
       \item For any $L>0$, $B(x_0,L)\cap\{\eta_r>0\}$ is relatively compact in $\mathcal{R}_X$. Here $x_0\in\mathcal{R}_X$ is a center of the soliton.
       \item 
       \begin{align*}
           \int_{\mathcal{R}_X\cap B(x_0,A)\cap\{|\nabla_g\eta_r|\neq0\}}dg\le \int_{\mathcal{R}_X\cap B(x_0,A)\cap\{0\le \eta_r<1\}}dg\le Cr^{4-\sigma},
       \end{align*}
       where $C=C(A,W,\sigma)$.
       \item $\displaystyle \mathcal{R}_X=\bigcup_{r>0}\{\eta_r=1\}$.
   \end{enumerate}
    
\end{Proposition}

The proofs of the results of this subsection are found in Section 6.

\subsection{Global \texorpdfstring{$\nu$}{nu}-functional and Sobolev inequality}

In \cite{CN09}, Carrillo-Ni made a nice observation: on a smooth shrinker, Perelman's global $\mu$-functional at scale $1$ $\mu(g,1)$ is equal to the shrinker entropy $\mu_g$. This is tantamount to saying that the potential function $f$ is exactly the minimizer of Perelman's $\mathcal{W}$-functional at scale $1$. The consequence is a sharp logarithmic Sobolev inequality. More recently, Li-Wang \cite{LW20} pushed this result one step more, and proved that the shrinker entropy $\mu_g$ is equal to Perelman's $\nu$-functional on all scales. Hence, there is a global Sobolev inequality on each shrinker. We shall prove that the same holds for metric solitons.

\begin{Theorem}\label{thm: nu-functional}
    Let $(X,d,\nu)$ be the model of an $n$-dimensional noncollapsed $\mathbb{F}$-limit metric soliton and $(\mathcal{R}_X,g,f_0)$ its regular part. Let $W$ be the soliton entropy. Then we have
    \begin{align*}
        \nu\left(\mathcal{R}_X,g\right)=W,
    \end{align*}
    where $\nu$ is Perelman's local $\nu$-functional defined in Subsection 6.1.
\end{Theorem}

Note that according to the definition of the local $\nu$-functional, the test functions of $\nu\left(\mathcal{R}_X,g\right)$ must be compactly supported in $\mathcal{R}_X$. Thus the logarithmic Sobolev inequalities and the Sobolev inequality which follow from the theorem above can only be applied to functions compactly supported on $\mathcal{R}_X$.

\begin{Corollary}[Logarithmic Sobolev inequalities]\label{coro: log Sobolev}
    In the same settings as Theorem \ref{thm: nu-functional}, we have
    \begin{align}\label{eq: the-log-Sobolev}
        &\ \int_{\mathcal{R}_X} u^2\log u^2\,dg-\int_{\mathcal{R}_X} u^2\,dg\cdot \log\left(\int_{\mathcal{R}_X} u^2\,dg\right)
        \\\nonumber
        \le &\ \tau\int_{\mathcal{R}_X}\left(4|\nabla u|^2+Ru^2\right)\,dg-\left(W+n+\frac{n}{2}\log(4\pi\tau)\right)\int_{\mathcal{R}_X}u^2\,dg
    \end{align}
    for any $\tau>0$ and any $u\in C^{0,1}_c(\mathcal{R}_X)$.
\end{Corollary}

\begin{Corollary}[Sobolev inequality]\label{coro: Sobolev}
    In the same settings as Theorem \ref{thm: nu-functional}, we have
    \begin{align}\label{eq: the-uniform-Sobolev}
        \left(\int_{\mathcal{R}_X} |u|^{\frac{2n}{n-2}}\,dg\right)^{\frac{n-2}{n}}\le C(n)e^{-\frac{2W}{n}}\int_{\mathcal{R}_X}\left(4|\nabla u|^2+ Ru^2\right)\,dg
    \end{align}
    for any $u\in C^{0,1}_c(\mathcal{R}_X)$.
\end{Corollary}

The proofs of the results in this subsection are found in Section 7.

\subsection{Volume growth lower estimate}

Lastly, we shall prove an optimal volume growth lower bound for metric solitons, similar to the smooth case considered by Munteanu-Wang \cite{MW12} and Li-Wang \cite{LW20}. The following theorem shows that a noncompact metric soliton, similar to a Ricci shrinker, has at least linear volume growth. This volume growth  lower bound is obviously optimal since it is fulfilled by the standard cylinder $\mathbb{S}^{n-1}\times\mathbb{R}$.

\begin{Theorem}\label{thm:linear-volume-lower-bound} Let $(X,d,\nu)$ be the model of an $n$-dimensional noncollapsed $\mathbb{F}$-limit metric soliton and $(\mathcal{R}_X,g,f_0)$ its regular part. Let $x_0\in\mathcal{R}_X$ be a center of the soliton and $W$ the soliton entropy. If $X$ is noncompact (in the sense of infinite diameter), then there exists positive dimensional constant $c(n)$ such that
\[
|B(x_0,r)\cap \mathcal{R}_X|\ge c(n)e^{W}r \quad \text{ for all  } \quad r\ge c(n)^{-1}.
\]
\end{Theorem}

There is not much to be commented on the proof of Theorem \ref{thm:linear-volume-lower-bound}, since it follows from the argument of \cite{MW12,LW20}. However, we mention that the prerequisite tools that make this argument possible are the cut-off functions in Proposition \ref{prop:regular cutoff} and the global Sobolev inequality in Theorem \ref{thm: nu-functional}. The proof of Theorem \ref{thm:linear-volume-lower-bound} is presented in Section 8.
\vskip 2mm \vskip 0pt
\noindent\textbf{Acknowledgements.} The authors would like to thank Felix Schulze for asking the question of the scalar curvature lower bound on metric solitons and for some helpful discussions.

\section{Metric soliton as an \texorpdfstring{$\mathbb{F}$-limit}{F-limit}}

According to Bamler \cite[Definition 3.57]{Bam20b}, a metric soliton is no more than a self-similar metric flow; this notion is too general, and could  either be related to or have nothing to do with the Ricci flow. For this reason, we shall restrict our attention to those metric solitons that could possibly arise in the singularity analysis of the Ricci flow, namely, the ones which are $\mathbb{F}$-limits of noncollapsed sequences of smooth Ricci flows. In this section, we shall survey the background and basic properties of such metric solitons.

Let $\displaystyle\left\{(M^i,g^i_t,x_i)_{t\in(-T_i,0]}\right\}_{i=1}^\infty$ be a sequence of $n$-dimensional Ricci flows, where $x_i\in M^i$ is a fixed base point. To ensure that Bamler's theory \cite{Bam20a, Bam20b, Bam20c} is applicable, we assume that \emph{each Ricci flow in this sequence has bounded curvature within each compact time-interval} (but the bounds need not to be uniform), see the appendix of \cite{Bam21} for justifications.  Assume that the sequence satisfies the $(\tau_0,Y_0)$-noncollapsed condition: there are positive numbers $\tau_0$ and $Y_0$, such that
\begin{align*}
    \mathcal{N}_{x_i,0}(\tau_0)\ge -Y_0 \quad \text{ for all }\quad i \in\mathbb{N}.
\end{align*}
It appears that the exact values of $\tau_0$ and $Y_0$ become inconsequential if the limit is a metric soliton with finite entropy; the reason will be explained below.

According to \cite{Bam20b,Bam20c}, after passing to a subsequence, we have the following convergence
\begin{align}\label{def-limiting-condition}
    \left((M^i,g^i_t)_{t\in(-T_i,0]},(\nu_{x_i,0\,|\,t})_{t\in(-T_i,0]}\right)\xrightarrow[i\to\infty]{\quad\mathbb{F},\ \mathfrak{C}\quad}\big(\mathcal{X},(\nu_t)_{t\in(-T,0)}\big),
\end{align}
where $\mathfrak{C}$ is some correspondence, $\displaystyle T=\lim_{i\to \infty}T_i\in(0,+\infty]$ (in case $T=+\infty$, the convergence above is on compact intervals), $\mathcal{X}$ is a metric flow over $(-T,0]$ whose singular part $\mathcal{X}_{<0}\setminus\mathcal{R}_{<0}$ has space-time codimension no less than $4$, and $\nu_t$ is a conjugate heat flow satisfying
\begin{align}\label{eq: Hn-concentration}
    \operatorname{Var}(\nu_t)\le H_n|t|.
\end{align}
Here (and afterwards) $\operatorname{Var}$ represents the variance and $H_n$ stands for the constant $\frac{(n-1)\pi^2}{2}+4$. We point out that the finial slice $\mathcal{X}_0$ of $\mathcal{X}$ consists of one point only, and we will often regard $\mathcal{X}$ as a metric flow over the interval $(-T,0)$.

The regular part $\mathcal{R}=\mathcal{R}_{<0}$ of $\mathcal{X}$ is a Ricci flow space-time, and we shall always denote by $g_t$ the metric on $\mathcal{R}_t$ for $t\in(-T,0)$. Then, on the regular part, the conjugate heat flow $\nu_t$ can be written as ($\tau=-t$)
\begin{align}\label{nu-in-Regular-part}
    d\nu_{t}=(4\pi\tau)^{-\frac{n}{2}}e^{-f}dg_{t},
\end{align} 
where $f$ is a smooth function on $\mathcal{R}$. Since the singular part is a null-set, the Nash entropy is well-defined for $\nu_t$.
\begin{align*}
    \mathcal{N}_{\nu}(\tau):=\int_{\mathcal{R}_{-\tau}}fd\nu_{-\tau}-\frac{n}{2} \quad \text{ for all }\quad \tau\in(0,T).
\end{align*}

Similar to a consequence of Perelman's monotonicity formula, whenever $\mathcal{N}_\nu(\tau)\equiv\operatorname{const}$ for all $\tau\in(0,T)$, $\big(\mathcal{X},(\nu_t)_{t\in(-T,0)}\big)$ is a metric soliton, which is a self-similar metric flow in the sense of \cite[Definition 3.57]{Bam20b}; this type of metric solitons are that which shall be studied in this article. Henceforth, for a metric soliton obtained in this way we shall always define
\begin{align}\label{soliton-entropy}
    W:\equiv \mathcal{N}_\nu(\tau) \in(-\infty,0],\quad \tau\in(0,T), 
\end{align}
and call this number $W$ the \emph{soliton entropy}. The continuity of the Nash entropy with respect to the $\mathbb{F}$-convergence \cite[Theorem 2.10]{Bam20b} immediately implies that
\begin{align}
    \mathcal{N}_{x_i,0}(\tau)>W-1\quad \text{ for all } \tau \in(0,T) \text{ whenever $i$ is large enough. } 
\end{align}
Thus we see that the $(\tau_0,Y_0)$-noncollapsing assumption can be replaced by $(T/2,W-1)$-noncollapsing assumption, and we shall always refer to such a soliton as an \emph{$n$-dimensional noncollapsed $\mathbb{F}$-limit metric soliton} (instead of a $(\tau_0,Y_0)$-noncollapsed $\mathbb{F}$-limit as we did in \cite{CMZ23b}). 

At this point, it is helpful to recall the important properties of a metric soliton as summarized in \cite[Theorem 2.18,Theorem 15.69]{Bam20c}.

\begin{Theorem}[{\cite[Theorem 2.18, Theorem 15.69]{Bam20c}}, basic properties of the metric soliton]\label{thm: soliton-basic-properties}
    Let $\big(\mathcal{X},(\nu_t)_{t\in(-T,0)}\big)$ be the metric soliton arising in the scenario of this section. 
    \begin{enumerate}
        \item  On the metric part $\mathcal{R}$, we have
    \begin{gather}\label{eq: basic-properties}
        \Ric+\nabla^2f-\tfrac{1}{2\tau}g=0,
        \\\nonumber
        \tau\left(2\Delta f-|\nabla f|^2+R\right)+f-n\equiv W,
        \\\nonumber
        -\tau\left(|\nabla f|^2+R\right)+f\equiv W,
        \\\nonumber
        \partial_{\mathfrak{t}}f=|\nabla f|^2,
    \end{gather}
    where $f$ is defined in \eqref{nu-in-Regular-part} and $W$ is the soliton entropy defined in \eqref{soliton-entropy}.

\item The heat kernel $K$ satisfies the following identity on $\mathcal{R}$
\begin{align}\label{self-similarity-of-CHK}
    (\tau\partial_{\mathfrak{t}}-\tau\nabla f)_{x_1}K(x_1\,|\,x_0)+ (\tau\partial_{\mathfrak{t}}-\tau\nabla f)_{x_0}K(x_1\,|\,x_0)=\frac{n}{2}K(x_1\,|\,x_0).
\end{align}

\item There is a singular space $(X,d)$ of dimension $n$, a probability measure $\nu$ on $X$, and an identification 
$$\mathcal{X}_{<0}=X\times(-T,0),$$
such that the following holds for all $t\in(-T,0)$

\begin{enumerate}[(a)]
    \item $(\mathcal{X}_t,d_t)=(X\times\{t\},|t|^{1/2}d)$.
    \item $\mathcal{R}=\mathcal{R}_X\times(-T,0)$ and $\partial_{\mathfrak{t}}-\nabla f$ corresponds to the standard vector field on the second factor; here $\mathcal{R}_X$ is the regular part of $X$.
    \item $(\mathcal{R}_t,g_t)=(\mathcal{R}_X\times\{t\},|t|g)$, where $g$ is the Riemannian metric on $\mathcal{R}_X$.
    \item $\nu_t = \nu$.
\end{enumerate}
Moreover, there is a unique family of probability measures $(\nu'_{x\,|\,t})_{x\in X,\, t\ge 0}$ such that the tuple $(X,d,\nu,(\nu'_{x\,|\,t})_{x\in X,\, t\ge 0})$ is a model for $\big(\mathcal{X},(\nu_t)_{t\in(-T,0)}\big)$ corresponding to the same identification (in the sense of \cite[Definition 3.57]{Bam20b}).

\item The singular part of $(X,d,\nu)$ is a null set with codimension no less than $4$.
\end{enumerate}
\end{Theorem}

Henceforth, we shall call $(X,d,\nu)$ (where $\nu=\nu_{-1}$) the \emph{model} of the metric soliton and $(\mathcal{R}_X,g,f_0)$ (where $f_0=f(\cdot,-1)$) its \emph{regular part}. It is clear that a metric soliton is completely determined by the regular part of its model. The equations in \eqref{eq: basic-properties} are constantly applied in the proof of the main theorems, so it is helpful to restate them. 

\begin{Corollary}
Let $(X,d,\nu)$ be the model of an $n$-dimensional noncollapsed $\mathbb{F}$-limit metric soliton and $(\mathcal{R}_X,g,f_0)$ the regular part of the model. Then the following equations hold on $\mathcal{R}_X$
\begin{gather}\label{eq: basic-properties-2}
        \Ric+\nabla^2f_0=\frac{1}{2}g,
        \\\nonumber
        2\Delta f_0-|\nabla f_0|^2+R+f_0-n=W,
        \\\nonumber
        |\nabla f_0|^2+R=f_0- W.
    \end{gather}
\end{Corollary}

\section{Metric soliton as an ancient \texorpdfstring{$\mathbb{F}$-limit}{F-limit}}

It is well-known that a smooth skrinking gradient Ricci soliton always generates an ancient solution, and this fact is essential to the study of Ricci shrinkers. For instance, the positivity of the scalar curvature is a consequence of the ancientness \cite{Che09}. However, it is not immediately clear from the limiting definition \eqref{def-limiting-condition} that a metric soliton defined on a finite interval $(-T,0)$ can be extended to an ancient metric flow. Even if it could, it is not clear whether such ancient metric flow is a noncollapsed $\mathbb{F}$-limit of a sequence of Ricci flows. In this section, we shall apply a blow-up argument to verify this fact.

\begin{Theorem}\label{thm: metric soliton is ancient}
    Let $\big(\mathcal{X},(\nu_t)_{t\in(-T,0)}\big)$ be an $n$-dimensional noncollapsed $\mathbb{F}$-limit metric soliton defined over a finite interval $(-T,0)$. Then it can be extended to an ancient metric flow $\big(\mathcal{X},(\nu_t)_{t\in(-\infty,0)}\big)$, which is also a metric soliton with the same model. Further more, there is a sequence of $n$-dimensional smooth Ricci flows $\{(M^i,g^i_t,x_i)_{t\in[-T_i,0]}\}_{i=1}^\infty$ with $T_i\nearrow +\infty$ such that
    \begin{align}\label{eq:nash-entropy-converge}
        \lim_{i\to\infty}\mathcal{N}_{x_i,0}(T_i)=W,
    \end{align}
    where $W$ is the soliton entropy
, and that
    \begin{gather}\label{eq: limiting-definition-refined}
      \left((M^i,g^i_t)_{t\in[-T_i,0]},(\nu_{x_i,0\,|\,t})_{t\in[-T_i,0]}\right)\xrightarrow[i\to\infty]{\quad\mathbb{F}\quad}\big(\mathcal{X},(\nu_t)_{t\in(-\infty,0)}\big),  
    \end{gather}
    over compact intervals.
\end{Theorem}

The strategy is to show that the blow-up limit of $\big(\mathcal{X},(\nu_t)_{t\in(-T,0)}\big)$ is a metric soliton with the same model and defined over the interval $(-\infty,0)$. To begin with, the following lemma is a trivial fact from \cite[Theorem 7.4, Theorem 7.6]{Bam20b}.

\begin{Lemma}
There is a metric flow pair $\left(\mathcal{X}^\infty,(\nu^\infty_t)_{t<0}\right)$ defined over $(-\infty,0)$, such that
    \begin{align}\label{eq-blow-up-limit}
    \left(\mathcal{X}^k,(\nu^k_{t})_{t\in(-Tk^2,0)}\right)\xrightarrow[k\to\infty]{\qquad\mathbb{F}\qquad} \left(\mathcal{X}^\infty,(\nu^\infty_t)_{t<0}\right),
\end{align}
on compact time intervals, where $\mathcal{X}^k$ and $\nu^k_{t}$ are the scaling notations defined in \cite[Definition 6.55]{Bam20b}.
\end{Lemma}

The following proposition is the key ingredient of the proof of Theorem \ref{thm: metric soliton is ancient}

\begin{Proposition}
    The limit $\left(\mathcal{X}^\infty,(\nu^\infty_t)_{t<0}\right)$ in \eqref{eq-blow-up-limit} is a metric soliton defined over $(-\infty,0)$ which is isometric to $\big(\mathcal{X},(\nu_t)_{t\in(-T,0)}\big)$ on the interval $(-T,0)$.
\end{Proposition}

\begin{proof}

Let $(X,d,\nu)$ be the model of the metric soliton $\big(\mathcal{X},(\nu_t)_{t\in(-T,0)}\big)$ and  $(\mathcal{R}_X,g,f_0)$ its regular part. Without of loss of generality, we may assume that $T>1$. By Theorem \ref{thm: soliton-basic-properties}, the vector field $\partial_{\mathfrak{t}}-\nabla f$ is complete on $\mathcal{R}$, and the $1$-parameter family of maps generated by it
\begin{align*}
    \phi_t&\ : \mathcal{R}_X(\cong\mathcal{R}_{-1}) \to \mathcal{R},
    \\
    \phi_{-1}&\ =\operatorname{id},
    \\
    \tfrac{\partial}{\partial t}\phi_t&\ = \left(\partial_{\mathfrak{t}}-\nabla f \right)\circ \phi_t.
\end{align*}
defines isometries 
$$\phi_t: \big(\mathcal{R}_X,|t|g, f_0\big)\longrightarrow \big(\mathcal{R}_t,g_t,f(\cdot,t)\big),\quad t\in(-T,0),$$
which can be extended to isometries
\begin{align}\label{eq:canonical-form}
    \overline{\phi}_t: (X,|t|^{1/2}d,\nu)\longrightarrow(\mathcal{X}_t,\nu_t), \quad t\in(-T,0).
\end{align}
Indeed, the fact that $\phi_t$ is an isometry between $(\mathcal{R}_X,|t|g,f_0)$ and $\big(\mathcal{R}_t,g_t,f(\cdot,t)\big)$ can also be verified by direct computation:
\begin{gather*}
    \tfrac{\partial}{\partial t}|t|^{-1}\phi_t^*g_t =|t|^{-1}\phi_t^*\left(\mathcal{L}_{(\partial_{\mathfrak{t}}-\nabla f)}g_t\right)+|t|^{-2}\phi_t^*g_t=-2|t|^{-1}\phi_t^*\left(\Ric+\nabla^2f-\frac{1}{2|t|}g_t\right)=0,
    \\
    \tfrac{\partial}{\partial t}f\circ\phi_t = \left(\partial_{\mathfrak{t}}f-|\nabla f|^2\right)\circ\phi_t=0.
\end{gather*}
Here we have applied \eqref{eq: basic-properties}.

Next, we define
\begin{align*}
    \overline{\phi}_t^k : (X,d,\nu)\to \left(\mathcal{X}^k,(\nu^k_{t})_{t\in(-Tk^2,0)}\right)
\end{align*}
by letting
\begin{align*}
    \overline{\phi}_t^k:=\overline{\phi}_{k^{-2}t}
\end{align*}
for all $t\in (-Tk^2,0)$ and for all $k\in\mathbb{N}$; it is clear that each $\overline{\phi}_t^k$ is an isometry onto $\mathcal{X}^k_t$. For any interval of the form $I=[-S,-S^{-1}]$, where $S>1$, we may construct a correspondence between $\{\left(\mathcal{X}^k,(\nu^k_{t})_{t\in(-Tk^2,0)}\right)\}_{k>S}$ by
\begin{align*}
    \mathfrak{C}:=\left((X,|t|d)_{t\in I},\left(\overline{\phi}^{k,-1}_t\right)_{t\in I}\right),
\end{align*}
where $\overline{\phi}^{k,-1}_t$ is the inverse of $\overline{\phi}^{k}_t$. We observe that
\begin{align}\label{eq: blow-up-equivalence}
    \dist^{\mathfrak{C}}_{\mathbb{F}}\left(\left(\mathcal{X}^k_I,(\nu^k_{t})_{t\in I}\right),\left(\mathcal{X}^j_I,(\nu^j_{t})_{t\in I}\right)\right)\equiv 0\qquad\text{ for all }\qquad k,\ j>S.
\end{align}
To see this, we need only to verify the following claim. 
\\

\noindent\textbf{Claim. } \emph{For any $\lambda>1$,
\begin{align*}
\psi_t:=\overline{\phi}^{\lambda}_t\circ \overline{\phi}^{-1}_t: \mathcal{X}_t\to \mathcal{X}^\lambda_t,\quad t\in(-T,0)
\end{align*}
defines an isometry between $\left(\mathcal{X}^\lambda,(\nu^\lambda_{t})_{t\in(-T\lambda^2,0)}\right)$ and $\left(\mathcal{X},(\nu_{t})_{t\in(-T,0)}\right)$ over the interval $(-T,0)$. }

\begin{proof}[Proof of the claim] 
It is already clear that each $\psi_t$ is an isometry between metric spaces. We need only to show that it preserves the conjugate heat kernels. Let $(\nu'_{x\,|\, t})_{x\in X,\, t\le 0}$ be the family of probability measures given in Theorem \ref{thm: soliton-basic-properties}. Fix any $x=\phi_t(x')\in\mathcal{X}_t$ and any $-T<s< t<0$, we may compute according to \cite[Definition 3.57]{Bam20b}
\begin{align*}
    (\psi_s)_* \nu_{x\,|\,s}=&\ (\overline{\phi}^\lambda_t)_* \nu'_{x'\,|\,\log(s/t)}=(\overline{\phi}_{\lambda^{-2}t})_*\nu'_{x'\,|\,\log(\lambda^{-2}s/\lambda^{-2}t)}
    \\
    =&\ \nu_{\overline{\phi}_{\lambda^{-2}t}(x')\,|\,\lambda^{-2}s}=\nu^\lambda_{\psi_t(x)\,|\,s}.
\end{align*}
Thus, $\psi_t$ satisfies \cite[Definition 3.13]{Bam20b}; this proves the claim. 
    
\end{proof}

Finally, the proposition follows from \eqref{eq: blow-up-equivalence}.

\end{proof}

\begin{proof}[Proof of Theorem \ref{thm: metric soliton is ancient}]

Since the metric soliton is obtained as an $\mathbb{F}$-limit \eqref{def-limiting-condition}, for each $k$ we have that, 
\begin{align*}
     \left((M^i,k^2g^i_{k^{-2}t})_{t\in(-k^2T_i,0]},(\nu_{x_i,0\,|\,k^{-2}t})_{t\in(-k^2T_i,0]}\right)\xrightarrow[i\to\infty]{\qquad\mathbb{F}\qquad}\left(\mathcal{X}^k,(\nu^k_{t})_{t\in(-Tk^2,0)}\right).
\end{align*}
We may, by a diagonal argument, find a new sequence of $n$-dimensional Ricci flows with uniformly bounded Nash entropy, such that
    $$ \left((\overline{M}^i,\overline{g}^i_t)_{t\in(-\overline{T}_i,0]},(\overline{\nu}_{x_i,0\,|\,t})_{t\in(-\overline{T}_i,0]}\right)\xrightarrow[i\to\infty]{\qquad\mathbb{F}\qquad}\big(\mathcal{X}^\infty,(\nu^\infty_t)_{t<0}\big).$$
    Since the soliton entropy depends only on  (the regular part of) the model, we have that
    $$\mathcal{N}_{\nu^\infty}(\tau)\equiv W\quad \text{ for all }\quad \tau>0.$$
By the continuity of the Nash entropy with respect to the $\mathbb{F}$-convergence \cite[Theorem 2.10]{Bam20c}, we may, after discarding finitely many terms and modifying $\overline{T}_i$, ensure that \eqref{eq:nash-entropy-converge} holds; this finishes the proof of the theorem.
\end{proof}

Since for the Ricci flow $(M^i,g^i_t)_{t\in[-T_i,0]}$, the scalar curvature always satisfies
\begin{align}\label{trivial-mp-1111}
R_{g^i_t}\ge -\frac{n}{T_i}\qquad \text{ for all }\qquad t\in[-T_i/2,0]
\end{align}
by the maximum principle, and since the $\mathbb{F}$-convergence can be upgraded to the smooth convergence on the regular part of the limit \cite[Theorem 2.5]{Bam20c}, we immediately have the following corollary (which is contained in \cite[Theorem 2.18]{Bam20c}), generalizing \cite{Che09}.

\begin{Corollary}\label{Coro: scalar-positivity}
    Let $(X,d,\nu)$ be the model of any $n$-dimensional noncollapsed $\mathbb{F}$-limit metric soliton, and let $(\mathcal{R}_X,g,f_0)$ be its regular part. Then we have that 
    \begin{align*}
        R\ge 0 \quad \text{ everywhere on }\quad \mathcal{R}_X.
    \end{align*}
Furthermore, if $R=0$ is somewhere on $\mathcal{R}_X$, then $g$ is a Ricci-flat metric.
\end{Corollary}
\begin{proof}
    That $R$ is nonnegative follows from Theorem \ref{thm: metric soliton is ancient} and \eqref{trivial-mp-1111}. That $R$ is positive unless $g$ is Ricci flat follows from the strong maximim principle (which is but a local argument).
\end{proof}

\section{Center and the potential function}

It is well known \cite{CZ10} that the potential function of a smooth Ricci shrinker has exactly quadratic growth, and its minimum point (which may not be unique) is usually used as a center or an origin of the soliton. In like manner, it is also convenient to find a center for a metric soliton. Although, as we shall see below, on the regular part of the model of a metric soliton, the potential function has almost quadratic growth, yet it is much more convenient to pick an $H_n$-center of $\nu_t$ instead of a minimum point of the potential function as an origin.

\subsection{Center of metric soliton}

Let $\big(\mathcal{X},(\nu_t)_{t<0}\big)$ be an $n$-dimensional noncollapsed $\mathbb{F}$-limit metric soliton that arises as a limit in \eqref{def-limiting-condition}. Write $d\nu_t:=(4\pi\tau)^{-\frac{n}{2}}e^{-f}dg_t$ on the regular part $\mathcal{R}\subset\mathcal{X}$. By Theorem \ref{thm: metric soliton is ancient}, we may assume that the soliton is defined over the interval $(-\infty,0)$. Let $(X,d,\nu)$ be the model of the metric soliton and $(\mathcal{R}_X,g,f_0)$ its regular part. 

By \eqref{eq: Hn-concentration}, we have 
\begin{align*}
    \int_{\mathcal{X}_{-1}}\operatorname{Var}(\delta_x,\nu_{-1})d\nu_{-1}=\int_{\mathcal{X}_{-1}\times\mathcal{X}_{-1}}d_{-1}^2(x,y)\,d\nu_{-1}(x)d\nu_{-1}(y)=\operatorname{Var}(\nu_{-1})\le H_n.
\end{align*}
Thus, we can always find a point $x_0\in\mathcal{X}_{-1}$, such that
\begin{align}\label{space-time-H_n-Var}
    \operatorname{Var}(\delta_{x_0},\nu_{-1})\le H_n.
\end{align}
Since $\mathcal{R}_{-1}$ is of full measure and $\mathcal{X}_{-1}$ is its metric completion, we may, without loss of generality, assume that $x_0\in\mathcal{R}_{-1}$.  
For otherwise we may replace it with a point arbitrarily close by, such that \eqref{space-time-H_n-Var} still holds with a slightly larger bound (we shall not rewrite this bound, since the constant $H_n$ itself does not have much significance except that it is a dimensional constant). Note that $x_0$ can be identified with a point on $\mathcal{R}_X$ which we will also call $x_0$. Henceforth, the point $x_0$ defined by \eqref{space-time-H_n-Var} will be called a \emph{center (or origin) of the metric soliton}; a metric soliton may have multiple centers, but their distance from each other cannot exceed $\sqrt{2H_n}$.

Since $\partial_{\mathfrak{t}}-\nabla f$ is a complete vector field on $\mathcal{R}=\mathcal{R}_X\times(-\infty,0)$ corresponding to the second factor, we may find an integral curve $x_0(t)$ in $\mathcal{R}$ of  $\partial_{\mathfrak{t}}-\nabla f$ with $x_0(-1)=x_0$. Obviously, $x_0(t)\in\mathcal{R}_{t}$ for all $t\in(-\infty,0)$.  We can apply this self-similarity to verify the following facts, which indicate why $x_0$ is called a ``center'' by us.

\begin{Lemma}\label{lm: Hn-center-on-metric-soliton}
    \begin{align}
        \operatorname{Var}\big(\nu_t,\delta_{x_0(t)}\big)\le H_n|t| \quad \text{ for all }\quad t\in(-\infty,0).
    \end{align}
\end{Lemma}
\begin{proof}
    Note that, by Theorem \ref{thm: soliton-basic-properties}, $f$ is constant along the integral curvatures of $\partial_{\mathfrak{t}}-\nabla f$ since
    \begin{align*}
        \partial_{\mathfrak{t}}f-|\nabla f|^2=0.
    \end{align*}
Therefore, 
\begin{align*}
    \operatorname{Var}(\nu_t,\delta_{x_0(t)})=&\ \int_{\mathcal{X}_t}d^2_t\big(x_0(t),y\big)\,d\nu_t(y)=\int_{\mathcal{R}_t}d_t^2\big(x_0(t),y\big)(4\pi|t|)^{-\frac{n}{2}}e^{-f(y)}\,dg_{t}(y)
    \\
    =&\  \int_{\mathcal{R}_X}|t|d^2_{g}(x_0,y)(4\pi)^{-\frac{n}{2}}e^{-f_0(y)}dg(y)
    \\
    =&\ |t|\int_{\mathcal{R}_{-1}}d_{-1}^2(x_0,y)\,d\nu_{-1}(y)
    \\
    =&\ |t|\operatorname{Var}(\nu_{-1},\delta_{x_0}).
\end{align*}
Here we have applied the facts that $\mathcal{X}_t\setminus\mathcal{R}_t$ is a null set, that $\mathcal{X}_t$ is the metric completion of $\mathcal{R}_t$, and that the integrand 
$$d_t^2\big(x_0(t),y\big)(4\pi|t|)^{-\frac{n}{2}}e^{-f(y)}\,dg_{t}(y)$$
varies by scaling along integral curves of $\partial_{\mathfrak{t}}-\nabla f$.
    
\end{proof}

Similarly, we also have the following lemma. Its significance may not be clear at this point, yet it is an important ingredient of the proof of Theorem \ref{thm: chow-lu-yang estimate}.

\begin{Lemma}\label{lm: R-integral-const}
    We have
    \begin{align*}
        |t|\int_{\mathcal{R}_{t}}R_{g_{t}}\,d\nu_{x_0(t/2)\,|\,t}\equiv \operatorname{const}\quad \text{ for all }\quad t\in(-\infty,0).
    \end{align*}
\end{Lemma}
\begin{proof}
    Let us write the integral as 
    \begin{align*}
        \int_{\mathcal{R}_{t}}|t|R_{g_{t}}(y)K(x_0(t/2)\,|\,y)\,dg_{t}(y).
    \end{align*}
    We shall verify that the integrand is invariant along the integral curves of $\partial_{\mathfrak{t}}-\nabla f$. It is clear that both $|t|R_{g_t}$ and $|t|^{-\frac{n}{2}}dg_t$ satisfy this property. We need only to consider $|t|^{\frac{n}{2}}K(x_0(t/2)\,|\,y)$. Indeed, we have
    \begin{align*}
        &\ \left(\frac{\partial}{\partial t} -(\nabla f)_y\right)\left(|t|^{\frac{n}{2}}K(x_0(t/2)\,|\,y)\right)
        \\
        = &\ |t|^{\frac{n}{2}}\left(-\frac{n}{2|t|}K(x_0(t/2)\,|\,y)+(\partial_{\mathfrak{t}}-\nabla f)_yK(x_0(t/2)\,|\,y)+\left.\frac{1}{2}\partial_sK(x_0(s)\,|\,y)\right|_{s=\frac{t}{2}}\right)
        \\
        =&\ |t|^{\frac{n}{2}+1}\left(-\frac{n}{2}K(x_0(t/2)\,|\,y)+(\tau\partial_{\mathfrak{t}}-\tau\nabla f)_yK(x_0(t/2)\,|\,y)+\left.(\tau\partial_{\mathfrak{t}}-\tau\nabla f)_xK(x\,|\,y)\right|_{x=x_0(t/2)}\right)
        \\
        =&\ 0,
    \end{align*}
    where we have applied \eqref{self-similarity-of-CHK}. The lemma then follows in the same way as the previous lemma.
\end{proof}

\subsection{Estimates of the potential function}

Next, we prove the quadratic growth estimate for the potential function of a metric soliton. Note that the regular part is not a complete Riemannian manifold, so the second variation along a geodesic is not available. This is the reason why we have a weaker lower bound.

\begin{proof}[Proof of Theorem \ref{thm:potential-function}]
    Let us consider the limiting definition of the metric soliton given by Theorem \ref{thm: metric soliton is ancient}. Let $\{(M^i,g^i_t,x_i)_{t\in[-T_i,0]}\}_{i=1}^\infty$ be the sequence in \eqref{eq: limiting-definition-refined} that converges to the metric soliton $(\mathcal{X},(\nu_t)_{t<0})$ in the $\mathbb{F}$-sense. 
    
    The lower bound estimate of $f_0$ is simply a consequence of \cite[Lemma 15.9(a)]{Bam20c}. In fact, applying the second inequality of \cite[Lemma 15.9(a)]{Bam20c} with $(x_i',t_i')=(x_i,0)$ and $s=-1$, we immediately have
    \begin{align}\label{eq:BamcLm15.9}
    (4\pi)^{-\frac{n}{2}}e^{-f_0(x)}\le C(\varepsilon)\exp\left(-W-\frac{1}{8+\varepsilon}\big(d^g_{W_1}(\nu,\delta_x)\big)^2\right)\qquad \text{ for all }\quad x\in\mathcal{R}_X(\cong\mathcal{R}_{-1}).
    \end{align}
  On the other hand, the choice of $x_0$ \eqref{space-time-H_n-Var} implies that (recall that $\nu_{-1}=\nu$)
  \begin{align}\label{eq:BamcLm15.9-2}
  d_g(x,x_0)\le d^g_{W_1}(\nu,\delta_x)+d^g_{W_1}(\nu,\delta_{x_0})\le d^g_{W_1}(\nu,\delta_x)+\sqrt{H_n}\quad \text{ for all }\quad x\in\mathcal{R}_X.
\end{align}    
  The lower bound of $f_0$ follows from the combination of \eqref{eq:BamcLm15.9} and  \eqref{eq:BamcLm15.9-2}.

   Next, we consider the upper bound. Note that the positivity of the scalar curvature and the third equation in \eqref{eq: basic-properties-2} imply that
    \begin{align}\label{eq: potential-differential-ineq-for-the-upper-bound}
        |\nabla_gf_0|^2 \le f_0-W.
    \end{align}
    Thus, it remains to show that $f_0(x_0)-W$ is bounded from above by a constant depending only on $n$, and then the upper bound of $f_0$ follows from integration along almost-distance-realizing curves on $\mathcal{R}_X$. 
    
    Since $x_0\in\mathcal{R}_{-1}(\cong\mathcal{R}_X)$, where $\mathcal{R}\subset\mathcal{X}$ is the region of smooth convergence, we may, according to \cite[Theorem 9.31]{Bam20b}, find a sequence $z'_i\in M^i$, such that
    $$(z'_i,-1)\xrightarrow{\quad i\to\infty\quad}x_0.$$
This convergence is either via the local diffeomorphisms which define the local smooth convergence, or is a convergence of points within a correspondence. The points $(z_i',-1)$ are ``almost $H_n$-centers'' of $(x_i,0)$. To see this, we argue in the same way as the proof of \cite[Proposition 3.13]{Bam20a}. \eqref{space-time-H_n-Var} implies that
\begin{align}\label{eq: local-measure-near-center-on-the-soliton}
    \int_{\mathcal{R}_X\cap B(x_0,\sqrt{2H_n})} (4\pi)^{-\frac{n}{2}}e^{-f_0}\,dg\ge \frac{1}{2}.
\end{align}
By Fatou's lemma and the nature of local smooth convergence, we have
\begin{align*}
    \liminf_{i\to\infty}\nu_{x_i,0\,|\,-1}\left(B_{-1}\left(z_i',\sqrt{2H_n}\right)\right)\ge \int_{\mathcal{R}_X\cap B(x_0,\sqrt{2H_n})} (4\pi)^{-\frac{n}{2}}e^{-f_0}\,dg\ge \frac{1}{2}.
\end{align*}
Let $(z_i,-1)$ be an $H_n$-center of $(x_i,0)$ with respect to the Ricci flow $g^i_t$ for each $i$. Then, discarding finitely many terms if necessary, we may apply \cite[Proposition 3.13]{Bam20a} again, to obtain that
\begin{align}\label{near-the-Hn-center-1}
    d_{g^i_{-1}}(z_i,z_i')\le 2\sqrt{2H_n}\le 3\sqrt{H_n}\quad \text{ for all }\quad i.
\end{align}
(This is because, if \eqref{near-the-Hn-center-1} were false, then the total measure of $\nu_{x_i,0\,|\,-1}$ would exceed $1$). Consequently, we have
    \begin{align*}
          d^{g^i_{-1}}_{W_1}\left(\nu_{x_i,0\,|\,-1},\delta_{z_i'}\right)\le &\ d^{g^i_{-1}}_{W_1}\left(\nu_{x_i,0\,|\,-1},\delta_{z_i}\right)+d_{g^i_{-1}}(z_i',z_i)
        \\
        \le &\ 4\sqrt{H_n}.
    \end{align*}
    Taking into account \cite[Corollary 5.11]{Bam20a} and the scalar curvature lower bound \eqref{trivial-mp-1111}, we obtain
    \begin{align*}
        \mathcal{N}_{z_i',-1}(1)\le \mathcal{N}_{x_i,0}(2)+C(n)\quad \text{ whenever $i$ is large enough}.
    \end{align*}
    The fact that $\lim_{i\to\infty }\mathcal{N}_{x_i,0}(2)=W$ then implies that
    $$\mathcal{N}_{z_i',-1}(1)\le W+C(n)\quad\text{ whenever $i$ is large enough.}$$
    Thus, by \cite[Theorem 8.1]{Bam20a}, we have
    \begin{align}\label{eq: volume-upper-intermediate}
        \left|B_{g^i_{-1}}(z_i',\sqrt{2H_n})\right|\le C(n)\exp(\mathcal{N}_{z_i',-1}(2H_n))\le C(n)\exp(\mathcal{N}_{z_i',-1}(1))\le C(n)e^{W},
    \end{align}
    for all large $i$. Here we have also applied the scalar curvature bound \eqref{trivial-mp-1111}. Taking this to the limit and applying Fatou's lemma, we have
    \begin{align*}
        \left|\mathcal{R}_X\cap B(x_0,\sqrt{2H_n})\right|\le C(n)e^W.
    \end{align*}
This estimate, combined with \eqref{eq: local-measure-near-center-on-the-soliton}, implies that
\begin{align*}
    \frac{1}{2}\le \sup_{\mathcal{R}_X\cap B(x_0,\sqrt{2H_n})}(4\pi)^{-\frac{n}{2}}e^{-f_0}\cdot\left|\mathcal{R}_X\cap B(x_0,\sqrt{2H_n})\right|\le C(n)\sup_{\mathcal{R}_X\cap B(x_0,\sqrt{2H_n})}e^{-f_0+W}.
\end{align*}
Hence
$$\inf_{\mathcal{R}_X\cap B(x_0,\sqrt{2H_n})}(f_0-W)\le C(n),$$
and by \eqref{eq: potential-differential-ineq-for-the-upper-bound},
$$f_0(x_0)-W\le C(n).$$
This finishes the proof of the upper bound.
\end{proof}

In the proof above, we have already shown the volume upper estimate Corollary \ref{volume-upper-1}.

\begin{proof}[Proof of Corollary \ref{volume-upper-1}]
In fact, when proving \eqref{eq: volume-upper-intermediate}, we can apply \cite[Theorem 8.1]{Bam20a} to obtain a stronger inequality
\begin{align*}
    \left|B_{-1}(z_i',r)\right|\le C(n)e^Wr^n,\quad\text{ for all }\quad 1\le r^2\le T_i/4,
\end{align*}
whenever $i$ is sufficiently large; taking this to the limit, the corollary follows.

\end{proof}

A similar argument can be applied to estimate the volume ratio at each point. Nevertheless, the constant becomes worse if the base point is not a center. We shall leave the proof of Corollary \ref{volume-upper-2}  to the reader.

\subsection{Scalar curvature near a center}

In this subsection, we shall estimate the constant in Lemma \ref{lm: R-integral-const} in terms of 
\begin{align*}c_R:=&\ \int_{\mathcal{R}_X}R_g(4\pi)^{-\frac{n}{2}}e^{-f_0}\,dg
\end{align*}
and the soliton entropy $W$ ($c_R$ is the constant defined in \eqref{eq:scalar-integral-positivity}). Note that $c_R\ge 0$ because of Corollary \ref{Coro: scalar-positivity}, and $c_R\le\frac{n}{2}$ due to \cite[Proposition 5.13]{Bam20a}, Fatou's lemma, and the nature of local smooth convergence on the regular part. The following lemma shall be used in the proof of Theorem \ref{thm: chow-lu-yang estimate}. 

\begin{Lemma}\label{lm: lower-estimate-of-R-central-integral}
    Whenever $c_R>0$, we have
    \begin{align*}
          |2t|\int_{\mathcal{R}_{2t}}R_{g_{2t}}\,d\nu_{x_0(t)\,|\,2t}\ge c(n,c_R,W)>0\quad \text{ for all }\quad t\in(-\infty,0).
    \end{align*}
\end{Lemma}

\begin{proof}
    By Lemma \eqref{lm: R-integral-const}, we need only to prove the lemma for $t=-1$. We argue by contradiction. Let $\{(\mathcal{X}^i,(\nu^i_t)_{t<0})\}_{i=1}^\infty$ be a sequence of counterexamples (with models $(X^i,d^i,\nu^i)$ and regular parts $(\mathcal{R}_X^i,g^i,f_0^i)$). Namely, we have
    \begin{align*}
        c_R^i\ge c_0>0,\quad W^i\ge W_0>-\infty, \quad \text{ for each }i,
    \end{align*}
    but
    \begin{align}\label{eq: integral-converging-to-zero-contradictory-assumption}
        \int_{\mathcal{R}^i_{-2}}R_{g^i_{-2}}\,d\nu_{x_0^i\,|\,-2}\searrow 0.
    \end{align}

    We may now extract a subsequence which $\mathbb{F}$-converges to a metric flow $(\mathcal{X}^\infty,(\nu^\infty_t)_{t<0})$. We also assume that the convergence is time-wise at both $t=-1$ and $t=-2$. It is easy to see that the limit is also a noncollapsed $\mathbb{F}$-limit metric soliton. Let $(X^\infty,d^\infty,\nu^\infty)$ be its model and $(\mathcal{R}_X^\infty,g^\infty,f_0^\infty)$ the regular part. 

    By \cite{CMZ23b}, we actually have a local smooth convergence from $(\mathcal{R}_X^i,g^i,f_0^i)$ to $(\mathcal{R}_X^\infty,g^\infty,f_0^\infty)$. Combining the quadratic estimate of the potential function (Theorem \ref{thm:potential-function}), the quadratic upper bound of $R$ given by the third equation in \eqref{eq: basic-properties-2}, and the volume upper bound Corollary \ref{volume-upper-1}, we know that the integrals  
$$c_R^i:= \int_{\mathcal{R}_X^i}R_{g^i}(4\pi)^{-\frac{n}{2}}e^{-f_0^i}\,dg^i$$
are uniformly negligible outside uniformly large compact sets. This then implies that
\begin{align*}
    \int_{\mathcal{R}_X^\infty}R_{g^\infty}(4\pi)^{-\frac{n}{2}}e^{-f_0^\infty}\,dg^\infty=\lim_{i\to\infty}c_R^i\ge c_0>0.
\end{align*}
Thus, the scalar curvature is positive everywhere on $\mathcal{R}^\infty\subset\mathcal{X}^\infty$ by Corollary \ref{Coro: scalar-positivity}.

On the other hand, by Lemma \ref{lm: Hn-center-on-metric-soliton} and \cite[Theorem 6.49]{Bam20b}, the conjugate heat kernels $(\nu_{x^i_0\,|\,t})_{t<-1}$ converge to a conjugate heat flow $(\mu^\infty_t)_{t<-1}$ on $\mathcal{X}^\infty$. The convergence can also be regarded as  a local smooth one according to \cite[Theorem 9.31]{Bam20c}. By Fatou's lemma and \eqref{eq: integral-converging-to-zero-contradictory-assumption}, we have
$$\int_{\mathcal{R}^\infty_{-2}}R_{g^\infty_{-2}}\,d\mu^\infty_{-2}=0.$$
Being a conjugate heat flow, $\mu^\infty_{-2}$ is fully-supported, which contradicts the fact that $R_{g^\infty}$ is positive everywhere on $\mathcal{R}^\infty$.

\end{proof}

\section{A lower bound for the scalar curvature}

We shall prove Theorem \ref{thm: chow-lu-yang estimate} in this section. Consider a non-Ricci-flat metric soliton in the statement of the theorem. By Lemma \ref{lm: R-integral-const} and Lemma \ref{lm: lower-estimate-of-R-central-integral}, we fix the constant $c_0$ by
\begin{align}\label{limiting-scalar-integral lower-bound}
    |2t|\int_{\mathcal{R}_{2t}}R_{g_{2t}}\,d\nu_{x_0(t)\,|\,2t}\equiv 2c_0\ge 2c(n,W,c_R)>0 \quad\text{ for all }\quad t\in(-\infty, 0). 
\end{align}
By Theorem \ref{thm: metric soliton is ancient}, we pick a sequence of $n$-dimensional Ricci flows $\{(M^i,g_{i,t})_{t\in[-T_i,0]}\}_{i=1}^\infty$ satisfying \eqref{eq:nash-entropy-converge} and \eqref{eq: limiting-definition-refined}. In particular, by discarding finitely many terms, we assume
\begin{align}\label{eq:noncollapsing-assumption-CLY-section}
    \mathcal{N}_{x_i,0}(\tau)\ge W-1\quad \text{ for all }\quad \tau\in(0,T_i].
\end{align}
By the standard scalar curvature lower estimate \eqref{trivial-mp-1111}, we may, after replacing each $T_i$ with $T_i/2n$, assume that for each $i$,
\begin{align}\label{Sc-almost-positivity}
    R_{g^i_t}\ge -\frac{1}{T_i} \quad \text{ on }\quad M^i\times[-T_i,0].
\end{align}
Let $z_{i,t}$ ($t\in[-T_i,0)$) be the points defined by the property that
\begin{align*}
    (z_{i,t},t) \text{ \emph{is an $H_n$-center of $(x_i,0)$ with respect to the Ricci flow $g_{i,t}$.}}
\end{align*}

\subsection{Local lower bound of the scalar curvature near \texorpdfstring{$H_n$}{Hn}-centers}

In order to apply the argument in \cite{CCMZ23}, we will need to obtain a uniform lower bound near the $H_n$-centers of the limiting sequence. Precisely, we shall prove the following theorem in this subsection.

\begin{Theorem}\label{thm: sequential-local-lower-bound-along-Hn-centers}
    By reducing $T_i$ if necessary, we have $T_i\nearrow +\infty$ satisfying the following property. For any $D>0$, there is a positive number $a=a(D)$ (depending also on the constant $c_0$ in \eqref{limiting-scalar-integral lower-bound}), such that whenever $i$ is large enough, we have
    \begin{align*}
        (1+|t|)R_{g_{i,t}}(x)\ge a\quad \text{ for all } \quad x\in B_{g_{i,t}}\left(z_{i,t},D\sqrt{1+|t|}\right) \quad \text{ and }\quad t\in [-T_i,0).
    \end{align*}
\end{Theorem}

The key technique of the subsection is presented in the proposition below. 

\begin{Proposition}\label{prop:local-lower-bound-1}
    For any $D>0$ there is a positive number $a=a(D)$ (depending also on the constant $c_0$ in \eqref{limiting-scalar-integral lower-bound}) with the following property. For any $\tau>0$ and any $\alpha\in(0,1/2]$, if $i\ge \underline{I}(D,c_0,\tau)$, then
    \begin{align*}
        |t|R_{g_{i,t}}(x)\ge a \quad \text{ for all } \quad x\in B_{g_{i,t}}(z_{i,t},D\sqrt{|t|}) \text{ and }\quad t\in [-(1+\alpha)\tau,-(1-\alpha)\tau].
    \end{align*}
\end{Proposition}
\begin{proof}
    Let us fix $\tau>0$ and $\alpha\in(0,1/2]$. Assume that $i$ is large enough such that $T_i\gg \tau$. By \cite[Theorem 2.5]{Bam20c} and  \cite[Theorem 9.31]{Bam20a}, we can find a sequence of points $z_i'\in M^i$, such that
    \begin{align*}
        (z_i',-2\tau)\xrightarrow{\quad i\to\infty\quad}x_0(-2\tau).
    \end{align*}
The convergence of points is either via the diffeomorphisms in the definition of the local smooth convergence, or understood to be the convergence of points within a correspondence. Note that, arguing as the proof of \eqref{near-the-Hn-center-1} (at scale $\sqrt{\tau}$ instead of $1$), we also have
    \begin{align}\label{near-the-Hn-center}
        \dist_{g_{i,-2\tau}}(z_{i,-2\tau},z'_i)\le C(n)\sqrt{\tau} \quad\text{ for all $i$ large enough}.
    \end{align}

Since, by \cite[Theorem 9.31]{Bam20a}, we have
$$R_{g_{i,-4\tau}}\longrightarrow R_{g_{-4\tau}},\qquad \qquad K^i(z_i',-2\tau\,|\,\cdot,-4\tau)\longrightarrow  K(x_0(-2\tau)\, |\ \cdot\ )\bigg|_{\mathcal{R}_{-4\tau}},  $$
locally smoothly on $\mathcal{R}_{-4\tau}$, Fatou's lemma and \eqref{Sc-almost-positivity} imply that
\begin{align*}
    \liminf_{i\to\infty}4\tau\int_{M^i}\left(R_{g^i_{-4\tau}}+\frac{1}{T_i}\right)\,d\nu_{z'_i,-2\tau\,|\,-4\tau}\ge 4\tau\int_{\mathcal{R}_{-2\tau}}R_{g_{-4\tau}}d\nu_{x_0(-2\tau)\,|\,-4\tau}=2c_0.
\end{align*}
It then follows that
\begin{align}\label{H_n-center-one-point-bound}
    \int_{M^i} R_{g^i_{-4\tau}}d\nu_{z'_i,-2\tau\,|\,-4\tau}\ge \frac{c_0}{4\tau}\quad\text{ for all $i$ large enough.}
\end{align}

Next, we define
\begin{align*}
    u_i: M^i\times[-4\tau,-2\tau] &\ \to \mathbb{R}_+,
    \\
    (x,t)&\ \mapsto \int_{M^i} R_{g^i_{-4\tau}}d\nu_{x,t\,|\,-4\tau}+\frac{1}{T_i}.
\end{align*}
    Note that this is a positive solution to the heat equation. While the scalar curvature is a supersolution to the heat equation,  we always have $R_{g_{i,t}}(x)\ge u_i(x,t)-\frac{1}{T_i}$ for $t\ge -4\tau$.  \cite[Proposition 5.13]{Bam20a} provides an upper bound for $u_i$, namely,
    \begin{align*}
        0<u_i\le \frac{1}{\tau}\left(\frac{n}{2}+\frac{\tau}{T_i}\right)\quad \text{ on }\quad M^i\times[-3\tau,-2\tau].
    \end{align*}
    Since, by \eqref{H_n-center-one-point-bound}, $u_i(z_i',-2\tau)\ge \frac{1}{\tau}\left(\frac{c_0}{4}+\frac{\tau}{T_i}\right)$, we can apply \cite[Theorem 4.1]{Bam20a} to $u_i$ on $M^i\times[-3\tau,-2\tau]$ to obtain a local lower bound.  For any $D'>0$, the following holds whenever $i$ is large enough,
    \begin{align*}
        u_i(x,-2\tau)\ge&\ \frac{1}{\tau}\left(\frac{n}{2}+\frac{\tau}{T_i}\right)\Phi\left(\Phi^{-1}\left(\frac{\frac{c_0}{4}+\frac{\tau}{T_i}}{\frac{n}{2}+\frac{\tau}{T_i}}\right)-D'\right)
        \\
        \ge &\ \frac{1}{\tau}\left(\frac{n}{2}+\frac{\tau}{T_i}\right)\Phi\left(\Phi^{-1}\left(\frac{c_0}{2n}\right)-D'\right)\quad\text{ for all}\quad x\in B_{g_{i,-2\tau}}(z_i',D'\sqrt{\tau}).
    \end{align*}
 Here $\Phi$ is the function defined in \cite[\S 4.1]{Bam20a}   In view of \eqref{near-the-Hn-center}, this can also be written as
    \begin{align}\label{local-estimate-in-space}
        R_{g_{i,-2\tau}}(x)\ge \frac{1}{\tau}\left(\frac{n}{2}+\frac{\tau}{T_i}\right)\Phi\left(\Phi^{-1}\left(\frac{c_0}{2n}\right)-D'-C(n)\right)-\frac{1}{T_i}\\\nonumber
        \text{ for all}\quad x\in B_{g_{i,-2\tau}}(z_{i,-2\tau},D'\sqrt{\tau}).
    \end{align}

    Finally, it remains to extend the local estimate in space \eqref{local-estimate-in-space} to a space-time neighborhood. To this end, let us fix an arbitrary point $(x,t)\in M^i\times[-(1+\alpha)\tau,-(1-\alpha)\tau]$ such that $x\in B_{g_{i,t}}(z_{i,t},D\sqrt{|t|})$. Recall that $\alpha\in(0,1/2]$. Let $(z',-2\tau)$ be an $H_n$-center of $(x,t)$ and let $(z'',-2\tau)$ be an $H_n$-center of $(z_{i,t},t)$. Since, by \cite[Lemma 6.3]{CMZ23a}, we have $\dist_{g_{i,-2\tau}}(z_{i,-2\tau},z'')\le C_0(n)\sqrt{\tau}$, it holds that
    \begin{align*}
        \dist_{g_{i,-2\tau}}(z_{i,-2\tau},z')\le  &\ \dist_{g_{i,-2\tau}}(z_{i,-2\tau},z'')+\dist_{g_{i,-2\tau}}(z',z'')
        \\
         \le  &\ C_0(n)\sqrt{\tau}+\dist_{W_1}^{g_{i,-2\tau}}(\delta_{z'},\nu_{x,t\,|\,-2\tau})+\dist_{W_1}^{g_{i,-2\tau}}(\delta_{z''},\nu_{z_{i,t},t\,|\,-2\tau})
        \\
        &\ +\dist_{W_1}^{g_{i,-2\tau}}(\nu_{z_{i,t},t\,|\,-2\tau},\nu_{x,t\,|\,-2\tau})
        \\
     \le&\ (C_1(n)+D)\sqrt{\tau}.
    \end{align*}
In combination with \cite[Proposition 3.13]{Bam20a}, we have if $D':= D+C_1(n)+\sqrt{2H_n}$, then
\begin{align}\label{local-measure-lower-bound}
    \nu_{x,t\,|\,-2\tau}\left(B_{g_{i,-2\tau}}\left(z_{i,-2\tau},D'\sqrt{\tau}\right)\right)\ge \tfrac{1}{2}.
\end{align}
Taking into account \eqref{local-estimate-in-space}, we have the following estimate whenever $i$ is large enough.
\begin{align*}
    \int_{M^i}R_{g_{i,-2\tau}}\,d\nu_{x,t\,|\,-2\tau}=&\  \int_{B_{g_{i,-2\tau}}\left(z_{i,-2\tau},D'\sqrt{\tau}\right)}R_{g_{i,-2\tau}}\,d\nu_{x,t\,|\,-2\tau}
    \\
    &\ +\int_{M^i\setminus B_{g_{i,-2\tau}}\left(z_{i,-2\tau},D'\sqrt{\tau}\right)}R_{g_{i,-2\tau}}\,d\nu_{x,t\,|\,-2\tau}
    \\
    \ge &\ \frac{1}{2}\left(\frac{1}{\tau}\left(\frac{n}{2}+\frac{\tau}{T_i}\right)\Phi\left(\Phi^{-1}\left(\frac{c_0}{2n}\right)-D'-C(n)\right)-\frac{1}{T_i}\right)-\frac{1}{2}\cdot\frac{1}{T_i}
    \\
    =&\ \frac{1}{2\tau}\left(\frac{n}{2}+\frac{\tau}{T_i}\right)\Phi\left(\Phi^{-1}\left(\frac{c_0}{2n}\right)-D-C_2(n)\right)-\frac{1}{T_i}.
\end{align*}
    Being a supersolution to the heat equation, $R$ always satisfies
    \begin{align*}
        R_{g_{i,t}}(x)\ge \int_{M^i}R_{g_{i,-2\tau}}\,d\nu_{x,t\,|\,-2\tau}.
    \end{align*}
    This finishes the proof of the proposition.
\end{proof}

Finally, to prove Theorem \ref{thm: sequential-local-lower-bound-along-Hn-centers}, it remains only to deal with the scalar curvature near the $0$-time-slices. This is covered by the following Proposition.

\begin{Proposition}
    For any $D>0$ there is a positive number $a=a(D)$ (depending also on the constant $c_0$ in \eqref{limiting-scalar-integral lower-bound}) with the following property. If $i\ge \underline{I}(D,c_0)$, then
    \begin{align*}
        R_{g_{i,t}}(x)\ge a \quad \text{ for all } \quad x\in B_{g_{i,t}}(z_{i,t},D) \text{ and }\quad t\in [-1,0).
    \end{align*}
\end{Proposition}
\begin{proof}
    The proof is a simpler version of the previous proposition. Fix $\tau=2$ in Proposition \ref{prop:local-lower-bound-1}. Then for any $D'>0$ and $i\ge \underline{I}(D',c_0)$, we have that
    \begin{align*}
        R_{g_{i,-2}}(x)\ge a(D')\quad \text{ for all }\quad x\in B_{g_{i,-2}}(z_{i,-2},D').
    \end{align*}
    Arguing in the same way as the proof of \eqref{local-measure-lower-bound}, we have that, for any $(x,t)\in B_{g_{i,t}}(z_{i,t},D)$, where $D>0$ and $t\in[-1,0)$, if $D'=D+C(n)$, then
    \begin{align*}
        \nu_{x,t\,|\,-2}\left(B_{g_{i,-2}}(z_{i,-2},D')\right)\ge \frac{1}{2}.
    \end{align*}
    Thus,
    \begin{align*}
        \int_{M_i}R_{g_{i,-2}}\, d \nu_{x,t\,|\,-2}=&\  \int_{B_{g_{i,-2}}(z_{i,-2},D')}R_{g_{i,-2}}\, d \nu_{x,t\,|\,-2}+\int_{M^i\setminus B_{g_{i,-2}}(z_{i,-2},D')}R_{g_{i,-2}}\, d \nu_{x,t\,|\,-2}
        \\
        \ge &\ \frac{1}{2}a(D+C(n))-\frac{1}{2T_i}.
    \end{align*}
    This finishes the proof in the same way as Proposition \ref{prop:local-lower-bound-1}.
\end{proof}

\subsection{Proof of the scalar curvature lower bound}

In this subsection, we conclude the proof of Theorem \ref{thm: chow-lu-yang estimate} following the argument of \cite{CCMZ23}. Let us recall the following conjugate heat kernel estimate \cite[Theorem 7.2]{Bam20a}.
\begin{align*}
      (4\pi|t|)^{-\frac{n}{2}}e^{-f_i(x,t)}:=K_i(x_i,0\,|\,x,t)\le \frac{C(n)}{(4\pi|t|)^{\frac{n}{2}}}\exp\left(-W-\frac{\dist^2_{g_{i,t}}(z_{i,t},x)}{9|t|}\right)
\end{align*}
for all $(x,t)\in M^i\times[-T_i,0)$. Note that we have applied \eqref{eq:noncollapsing-assumption-CLY-section} and  \eqref{Sc-almost-positivity}.   This estimate is equivalent to
\begin{align}\label{CHK-upper}
    f_i(x,t)\ge -C_0+1+W+\frac{\dist^2_{g_{i,t}}(z_{i,t},x)}{9|t|}\quad \text{ for all }\quad (x,t)\in M^i\times[-T_i,0),
\end{align}
where $C_0$ is a dimensional constant. Let $$\tau=-t=|t|$$
throughout the rest of the subsection. Define the auxiliary function
\begin{align*}
    \rho_i(x,t):=\sqrt{1+\tau}\left(f_i(x,t)+C_0-W\right).
\end{align*}
Obviously, \eqref{CHK-upper} implies that
\begin{align}\label{rho-lower}
    \rho_i(x,t)\ge \sqrt{1+\tau} \left(1 + \frac{\dist^2_{g_{i,t}}(z_{i,t},x)}{9\tau}\right)\quad\text{ for all }\quad (x,t)\in M^i\times[-T_i,0).
\end{align}

Next, we define the function with which we shall apply the maximum principle. Let $$A=A(n,W):=C_0 - W + \tfrac{n}{2}$$ and
\begin{align*}
    D=D(A,n)<+\infty,\quad \quad a_0=a(D)>0, \quad \quad c=c(a_0,D)>0
\end{align*}
to be fixed in the course of the proof; here $a(D)$ is the constant given by Theorem \ref{thm: sequential-local-lower-bound-along-Hn-centers}.
For each $i$ large enough, we define
\begin{align*}
    F_i(x,t):=\sqrt{1+\tau} \, R_{g_{i,t}}(x) - c \rho_i^{-1}(x,t) - 4cA\sqrt{1+\tau}\,\rho_i^{-2}(x,t),\quad (x,t)\in M^i\times[-T_i,0).
\end{align*}
Arguing in the same way as in \cite{CCMZ23}, we have:
\begin{Lemma}\label{lm: PDE-of-F}
The following differential inequality holds on $M^i\times[-T_i,0]$.
    \begin{align*}
     \tau\left(\frac{\partial}{\partial t}-\Delta_{g_{i,t}}\right) F_i \ge -\tfrac{1}{2}F_i + cA \sqrt{1+\tau}\,\rho_i^{-3}\left(\rho_i-8A\sqrt{1+\tau}\,\right)-\frac{1}{2\sqrt{1+\tau}}\cdot \frac{1}{T_i}.
\end{align*}
\end{Lemma}
\begin{proof}
    We shall suppress the indices in the proof of the Lemma. The conjugate heat equation $\partial_t f=-\Delta f+|\nabla f|^2-R+\frac{n}{2\tau}$ implies that
    $$\tau\left(\frac{\partial}{\partial t}-\Delta_{g_t}\right) f=f-\tfrac{n}{2}-w,$$
    where
    $$w:=\tau(2\Delta f-|\nabla f|^2+R)+f-n\le 0$$
    by Perelman's Harnack inequality \cite{Per02}. Thus
    $$\tau\left(\frac{\partial}{\partial t}-\Delta_{g_t}\right)\rho \ge -\frac{\tau}{2(1+\tau)}\rho+\sqrt{1+\tau}\left(f-\tfrac{n}{2}\right)\ge \tfrac{1}{2}\rho-A\sqrt{1+\tau},$$
    due to our choice of $A=C_0-W+\frac{n}{2}$ and \eqref{CHK-upper}.
    Then
    \[
    \tau\left(\frac{\partial}{\partial t}-\Delta_{g_t}\right)\rho^{-1}=-\tau\rho^{-2}\left(\tfrac{\partial}{\partial t}-\Delta_{g_t}\right)\rho-2\tau\rho^{-3}|\nabla \rho|^2\le -\tfrac{1}{2}\rho^{-1}+A\sqrt{1+\tau}\rho^{-2}.
    \]
    Similarly, we have
    \begin{align*}
        \tau\left(\frac{\partial}{\partial t}-\Delta_{g_t}\right)\rho^{-2}&\ \le -\rho^{-2}+2A\sqrt{1+\tau}\rho^{-3},
        \\
        \tau\left(\frac{\partial}{\partial t}-\Delta_{g_t}\right)\left(\sqrt{1+\tau}\rho^{-2}\right)&\ \le -\sqrt{1+\tau}\rho^{-2}+2A(1+\tau)\rho^{-3}.
    \end{align*}
    In combination, we have
    \[
    \tau\left(\frac{\partial}{\partial t}-\Delta_{g_t}\right)\left(\rho^{-1}+4A\sqrt{1+\tau}\rho^{-2}\right)\le -\tfrac{1}{2}\left(\rho^{-1}+4A\sqrt{1+\tau}\rho^{-2}\right)-A\sqrt{1+\tau}\rho^{-3}\left(\rho-8A\sqrt{1+\tau}\right).
    \]
    The lemma follows from combining the inequality above with the following fact.
    \begin{align*}
        \tau\left(\frac{\partial}{\partial t}-\Delta_{g_t}\right)\left(\sqrt{1+\tau}R\right)&\ =-\frac{\tau}{2\sqrt{1+\tau}}R+2\tau\sqrt{1+\tau}|\Ric|^2
        \\
        &\ \ge -\frac{\tau}{2\sqrt{1+\tau}}\left(R+\frac{1}{T_i}\right)+\frac{\tau}{2\sqrt{1+\tau}}\cdot\frac{1}{T_i}
        \\
        &\ \ge -\tfrac{1}{2}\sqrt{1+\tau}R-\tfrac{1}{2}\sqrt{1+\tau}\cdot\frac{1}{T_i}+\frac{\tau}{2\sqrt{1+\tau}}\cdot\frac{1}{T_i}
        \\
        &\ =-\tfrac{1}{2}\sqrt{1+\tau}R-\frac{1}{2\sqrt{1+\tau}}\cdot \frac{1}{T_i},
    \end{align*}
    where we have applied \eqref{Sc-almost-positivity}.
\end{proof}

The following boundary conditions can be easily checked.

\begin{Lemma}\label{lm:boundary-conditions}
If $$D>0,\quad c<\frac{a_0}{1+4A},\quad \text{ and }\quad i\ge \underline{I}(c,A,D),$$ then
    \begin{gather}
         \liminf_{t\to 0-}\left(\inf_{x\in M}F_i(x,t)\right) \ge -\frac{1}{T_i}, \label{boundary-t=0}
         \\
         \liminf_{x\to\infty}F_i(x,t) \ge -\frac{\sqrt{1+T_i}}{T_i}\label{boundary-spatial-infinity}
         \\
         \inf_{x\in M^i}F_i(x,-T_i) \ge -\frac{\sqrt{1+T_i}}{T_i}-\frac{c(1+4A)}{\sqrt{1+T_i}}. \label{boundary-t=tau}
    \end{gather}
\end{Lemma}
\begin{proof}
    We first prove \eqref{boundary-t=0}. We fix an $i$ that is large enough. Let $t\in(-\tau_i,0)$ be a fixed time. For any point $x\in M$, if $x\in B_{g_{i,t}}(z_{i,t},D\sqrt{1+\tau}\,)$, then Theorem \ref{thm: sequential-local-lower-bound-along-Hn-centers} and the choices of $a_0$ and $c$ imply that
\begin{align*}
    F_i(x,t)\ge \frac{a_0}{\sqrt{1+\tau}}-\frac{c(1+4A)}{\sqrt{1+\tau}}\ge0. 
\end{align*}
If $x\not\in B_{g_{i,t}}(z_{i,t},D\sqrt{1+\tau})$, then we have
$$\rho_i(x,t)\ge \sqrt{1+\tau}\left(1+\frac{D^2(1+\tau)}{\tau}\right)\ge \frac{D^2}{\tau},$$
and thus
\begin{align*}
    F_i(x,t)\ge -\frac{1}{T_i}\sqrt{1+\tau}-\frac{c\tau}{D^2}-\frac{4cA\tau^2\sqrt{1+\tau}}{D^4},
\end{align*}
where we have applied \eqref{Sc-almost-positivity}.
Therefore, we have
$$\inf_{x\in M}F_i(x,t)\ge -C(A,D,c)\tau-\frac{1}{T_i}\sqrt{1+\tau}\to -\frac{1}{T_i} \quad\text{ as }\quad t\to 0-.$$
This shows \eqref{boundary-t=0}.

\eqref{boundary-spatial-infinity} and \eqref{boundary-t=tau} both follow directly from the definition of $F^i$, the standard scalar curvature lower bound \eqref{Sc-almost-positivity}, and the lower bound of $\rho$ \eqref{rho-lower}.
\end{proof}

Finally, Theorem \ref{thm: chow-lu-yang estimate} follows from taking the proposition below to the limit.

\begin{Proposition}
    If
    \begin{align*}
        D\ge\underline{D}(A), \quad c<\frac{a_0}{1+4A},\quad \text{ and }\quad i\ge \underline{I}(c,A,D),
    \end{align*}
    then
    \begin{align*}
        \inf_{M^i\times [-T_i,0)}F_i\ge -2\cdot\left(\frac{\sqrt{1+T_i}}{T_i}+\frac{c(1+4A)}{\sqrt{1+T_i}}\right).
    \end{align*}
\end{Proposition}
\begin{proof}
    
Assume by contradition that for some $i$ large enough (such that Lemma \ref{lm:boundary-conditions} is applicable), we have
\begin{align*}
    \inf_{M^i\times[-\tau_i,0)}F_i< -2\cdot\left(\frac{\sqrt{1+T_i}}{T_i}+\frac{c(1+4A)}{\sqrt{1+T_i}}\right)<0.
\end{align*}
By the boundary conditions \eqref{boundary-t=0}---\eqref{boundary-t=tau}, the infimum must be attained at an interior point $(x_0,t_0)\in M^i\times (-T_i,0)$. Then, by Lemma \ref{lm: PDE-of-F},
\begin{align*}
    0\ge\tau\left(\frac{\partial}{\partial t}-\Delta_{g_{i,t}}\right) F_i\ge -\tfrac{1}{2}F_i + cA \sqrt{1+\tau}\,\rho_i^{-3}\left(\rho_i-8A\sqrt{1+\tau}\,\right)-\frac{1}{2\sqrt{1+\tau}}\cdot \frac{1}{T_i}\quad\text{at}\quad (x_0,t_0).
\end{align*}
Thus, the following holds true at the point $(x_0,t_0)$.
\begin{align*}
   cA \sqrt{1+\tau}\,\rho_i^{-3}\left(\rho_i-8A\sqrt{1+\tau}\,\right)&\ \le \frac{1}{2\sqrt{1+\tau}}\cdot \frac{1}{T_i}+\frac{1}{2}F_i
   \\
   &\le  \frac{1}{2\sqrt{1+\tau}}\cdot \frac{1}{T_i}-\frac{\sqrt{1+T_i}}{T_i}-\frac{c(1+4A)}{\sqrt{1+T_i}}
   \\
   &<0.
\end{align*}
Therefore,
\begin{align*}
    8A\sqrt{1+|t_0|}\ge \rho_i(x_0,t_0)\ge \sqrt{1+|t_0|} \left(1 + \frac{\dist^2_{g_{i,t_0}}(z_{i,t_0},x_0)}{9|t_0|}\right)
\end{align*}
It follows that
\begin{align*}
    \dist_{g_{i,t_0}}(z_{i,t_0},x_0)\le\sqrt{72A(1+|t_0|)}.
\end{align*}
Now, if we take $D\ge \sqrt{100A}$, then Theorem \ref{thm: sequential-local-lower-bound-along-Hn-centers} implies that
\begin{align*}
     F_i(x_0,t_0)\ge \frac{a_0}{\sqrt{1+|t_0|}}-\frac{c(1+4A)}{\sqrt{1+|t_0|}}\ge 0.
\end{align*}
This is a contradiction; the proof of the proposition is complete.

\end{proof}


\section{A local gap theorem}

In this section, we prove Theorem \ref{thm: local-gap} by constructing a cut-off function and applying the method of \cite{CMZ22}. Note that the cut-off function in Proposition \ref{prop:regular cutoff} is also very helpful in the proof of Theorem \ref{thm:linear-volume-lower-bound}. 

\subsection{Perelman's functionals}

In this subsection, we recall some of Perelman's functionals, particularly their local versions. A systematic treatment of Perelman's local functionals can be found in \cite{W18}. Let $M^n$ be any smooth manifold and $\Omega\Subset M$ any precompact open subset. For any Riemannian metric $g$ and real number $\tau>0$, define
\begin{align*}
    \mu(\Omega,g,\tau):=\inf\left\{\overline{\mathcal{W}}(g,u,\tau)\,\left|\,u\in C^\infty_c(\Omega),\ \int_Mu^2\,dg=1\right\}\right.,
\end{align*}
where
\begin{align*}
    \overline{\mathcal{W}}(g,u,\tau):=\int_M\left(\tau(4|\nabla u|^2+Ru^2)-u^2\log u^2\right)\,dg-\left(\frac{n}{2}\log(4\pi\tau)+n\right).
\end{align*}
The local $\nu$-functional is defined to be the infimum of the local $\mu$-functional (on all scales or up to a certain scale):
\begin{align*}
    \nu(\Omega,g)&\ :=\inf_{\tau>0}\nu(\Omega,g,\tau),
    \\
    \nu(\Omega,g,\tau)&\ :=\inf_{\tau\ge s>0}\nu(\Omega,g,s).
\end{align*}

It is clear that the local $\mu$-functional is a local logarithmic Sobolev constant.
\begin{Proposition}\label{prop: local-log-Sobolev-of-mu}
    If $\mu(\Omega,g,\tau)>-\infty$, then
    \begin{align*}
        \tau\int_M(4|\nabla u|^2+Ru^2)\,dg\le \int_M u^2\log u^2\,dg+\frac{n}{2}\log(4\pi\tau)+n-\mu(\Omega,g,\tau),
    \end{align*}
    for any $u\in C_{c}^{0,1}(\Omega)$ satisfying 
    $$\int_M u^2\,dg=1.$$
\end{Proposition}

A standard argument in analysis (cf. \cite{Zh10}) also implies the following Sobolev inequality.
\begin{Proposition}
    If $\nu_0:=\nu(\Omega,g,\tau)>-\infty$ and $R\ge R_{\min}$ on $\Omega$, then there is a positive number $c=c(n)$, such that \begin{align*}
\left(\int_M|u|^{\frac{2n}{n-2}}\,dg\right)^{\frac{n-2}{n}}
  \leq c e^{-\frac{2\nu_0}{n}-c\tau R_{\operatorname{min}}} \int_{\Omega} \bigg(4|\nabla u|^2
  +\left( R_{g_0}-R_{\operatorname{min}}+\frac{c}{\tau}\right)u^2 \bigg)\,dg, 
\end{align*}
for any $u\in C_c^{0,1}(\Omega)$.
\end{Proposition}

\subsection{A preliminary soliton entropy gap}

The following gap theorem of the soliton entropy is a straightforward consequence of Bamler's $\varepsilon$-regularity theorem. The critical idea of the proof of Theorem \ref{thm: local-gap} is to convert the local gap problem to the entropy gap problem.

\begin{Theorem}\label{thm-Nash-gap}
    There is a constant $\varepsilon=\varepsilon(n)$ with the following property. Let $\big(\mathcal{X},(\nu_t)_{t\in(-T,0)}\big)$ be a noncollapsed $\mathbb{F}$-limit metric soliton defined over interval $(-T,0)$, where $T\in(0,+\infty]$. Let $W$ be its soliton entropy defined in \eqref{soliton-entropy}. 
    If $W\ge-\varepsilon$, then $\big(\mathcal{X},(\nu_t)_{t\in(-T,0)}\big)$ is the flat Gaussian shrinker.
\end{Theorem}
\begin{proof}
    Let $\varepsilon(n)=\frac{1}{2}\varepsilon_n$, where $\varepsilon_n$ is the constant in \cite[Theorem 10.2]{Bam20a}. By Theorem \ref{thm: metric soliton is ancient}, there is a sequence $\{(M^i,g^i_t,x_i)_{t\in[-T_i,0]}\}_{i=1}^\infty$ with $T_i\nearrow +\infty$ such that \eqref{eq:nash-entropy-converge} and \eqref{eq: limiting-definition-refined} hold. By \cite[Theorem 10.2]{Bam20a} and \eqref{eq:nash-entropy-converge}, we have that
    \begin{align*}
        \left|\Rm_{g^i}\right|\le \frac{1}{\varepsilon^2 T_i}\quad \text{ on }\quad B_{g^i_0}\left(x_i,\varepsilon\sqrt{T_i}\right)\times[-\varepsilon^2 T_i,0].
    \end{align*}
    It follows that the metric soliton is a smooth and flat shrinking gradient Ricci soliton; this finishes the proof of the Theorem.
\end{proof}

\subsection{Construction of the cut-off functions}

In this subsection is the proof of Proposition \ref{prop:regular cutoff}. We shall construct cut-off functions for the regular part of a metric soliton. Precisely, we shall convert the cut-off function \cite[Lemma 15.27]{Bam20c} to one that can be applied to metric solitons. 

 Let $(\mathcal{X},\nu_t)$ be an $n$-dimensional noncollapsed $\mathbb{F}$-limit metric soliton. Denote by $(X,d,\nu)$ its model and $(\mathcal{R}_X,g,f_0)$ the regular part of the model. Let $x_0\in\mathcal{R}_X$ be a center of the soliton and $W$ the soliton entropy. Let $x_0(t)$ be the integral curve of $\partial_{\mathfrak{t}}-\nabla f$ with $x_0(-1)=x_0\in\mathcal{R}_{-1}$. By Theorem \ref{thm: metric soliton is ancient}, we may assume that the soliton is defined over $(-\infty,0)$, and that there is a sequence $\{(M^i,g^i_t,x_i)_{t\in[-T_i,0]}\}_{i=1}^\infty$ satisfying $T_i\nearrow +\infty$, \eqref{eq:nash-entropy-converge}, and \eqref{eq: limiting-definition-refined}. Since in \cite[Lemma 15.27]{Bam20c}, the $P^*$-parabolic neighborhood is defined using the limit of $\nu_{x_i,0\,|\,t}$, which is $\nu_t$ in our case, let us define
\begin{align*}
    P^*_\nu(A,-T):=\left\{x\in\mathcal{X}_{[-T,0)}\,\left|\,d^{-T}_{W_1}(\nu_{x\,|\,-T},\nu_{-T})<A\right\}\right.
\end{align*}

\begin{Lemma}\label{lm:P-star-parabolic}
    For any $x\in B_t\big(x_0(t),A\big)\subset\mathcal{X}_t$, where $t\in[-T,0)$, we have
    $$x\in P^*_\nu\left(A+\sqrt{H_nT},-T\right).$$
\end{Lemma}
\begin{proof}
    By Lemma \ref{lm: Hn-center-on-metric-soliton}, the fact $d_{W_1}\le \sqrt{\operatorname{Var}}$, and the monotonicity of the $W_1$-Wasserstein distance (cf. \cite[Proposition 4.17]{Bam20c}), we have
    \begin{align*}
        d^{-T}_{W_1}(\nu_{x\,|\,-T},\nu_{-T})&\ \le 
 d^{t}_{W_1}(\delta_x,\nu_{t})
 \\
 &\  \le 
 d_t(x_0(t),x)+d^{t}_{W_1}(\delta_{x_0(t)},\nu_{t})
 \\
 &\ \le A+\sqrt{H_nT}.
    \end{align*}
    This proves the lemma.
\end{proof}

\begin{proof}[Proof of Proposition \ref{prop:regular cutoff}]

For fixed $\sigma>0$ and $A>\underline{A}(n)$, Lemma \ref{lm:P-star-parabolic} implies that, 
\begin{align}\label{eq:P-star-parabolic}
    B_t(x_0(t),A\sqrt{|t|})\subset P^*_\nu(2A,-2)\quad \text{ for any }\quad t\in[-2,0).
\end{align}
Now, if $r\le \overline{r}(A,W,\sigma)$, we may consider the cut-off function $\eta_r$ defined by Bamler in \cite[Lemma 15.27]{Bam20c}. Note that in the statement thereof, $\Delta=4$, and the parameters $Y_0$ and $\tau_0$ are absorbed into the parameter $W$ in our case due to \eqref{eq:nash-entropy-converge}. 

By \eqref{eq:P-star-parabolic} and \cite[Lemma 15.27]{Bam20c}, we have
\begin{align*}
    &\ \int_{-2}^{-\frac{1}{2}}\left(\int_{\mathcal{R}_t\cap B_t(x_0(t),A\sqrt{|t|})\cap\{0\le \eta_r<1\}}dg_t\right)dt
    \\
    \le &\ \int_{-2}^0\int_{\{0<\tilde{r}_{\Rm}\le 2r\}\cap P^*(2A,-2)\cap\mathcal{R}_t}dg_tdt
    \\
    \le &\ Cr^{4-\sigma},
\end{align*}
where $C=C(A,W,\sigma)$. Thus, we may find a time $t\in[-2,-1/2]$, such that
    $$\int_{\mathcal{R}_t\cap B_t(x_0(t),A\sqrt{|t|})\cap\{0\le \eta_r<1\}}dg_t\le 2Cr^{4-\sigma}.$$
Since the geometry of $\mathcal{R}_t\subset\mathcal{X}_t$ and $\mathcal{R}_X\subset X$ differ only by a scaling with bounded factor $|t|\in[1/2,2]$, and since $B_t(x_0(t),A\sqrt{|t|})\subset\mathcal{X}_t$ corresponds to $B(x_0,A)\subset X$, we see that the time-$t$-slice of $\eta_r$, when identified as a function on $X$ via the isometry $\overline{\phi}_t$ defined in \eqref{eq:canonical-form}, satisfies Properties (1)---(4) in the statement of the proposition. 

Finally, to see Property (5), let us fix an arbitrary $x\in \mathcal{R}_X$ and let $x(t)$ ($t\in[-2,-1/2]$) be its trajectory. We may find a positive number $c$, such that for all $t\in[-2,-1/2]$, it holds that the parabolic neighborhood $B_t(x(t),2c)\times[t-4c^2,t+4c^2]\subset\mathcal{R}$ is unscathed and
\begin{align*}
|\Rm_{g_t}|\le (2c)^{-2}\quad \text{ on }\quad B_t(x(t),2c)\times[t-4c^2,t+4c^2]. 
\end{align*}
On the other hand, for each $t\in[-2,-1/2]$, we can also find $x_i'\in M^i$, such that
\begin{align*}
(x'_i,t)\xrightarrow{\quad i\to\infty\quad} x(t).
\end{align*}
This convergence is either via the local diffeomorphisms which define the smooth convergence, or is a convergence of points within a correspondence. By the local smooth convergence property \cite[Theorem 2.5]{Bam20c} at $x(t)$, we have that
$$r_{\Rm}(x_i',t)\ge c\qquad \text{ for all $i$ large enough},$$
where $r_{\Rm}$ is the curvature scale defined in \cite[Definition 10.1]{Bam20a}. Thus, we have $\tilde{r}_{\Rm}(x(t))\ge c$ for all $t\in[-2,-1/2]$ due to \cite[Lemma 15.16]{Bam20c}. By Bamler's definition of $\eta_r$, we have $\eta_r(x(t))=1$ for all $t\in[-2,-1/2]$ whenever $r<c/2$. Thus, in our choice of $\eta_r$ above, no matter which time slice it is, we always have $x\in \{\eta_r=1\}$ if $r<c/2$; this shows property (5).
     
\end{proof}

\subsection{Proof of the local gap theorem}

In this subsection, we prove Theorem \ref{thm: local-gap}. The settings are the same as that of the previous subsection. Let
$$\sigma>0,\quad0<\delta\le\overline{\delta}(n), \quad 0<r\le \overline{r}(\delta,W,\sigma)$$
to be determined. Let $\eta$ be a standard cutoff function on $\IR$ such that
\[
    \eta|_{(-\infty,1/2]}=1,\quad
    \eta|_{[1,\infty)}=0,\quad
    - 3\sqrt{\eta} \le \eta'\le 0.
\]
Define
\begin{align}\label{eq:cut-off-on-large-scale}
\phi_\delta(x) = \eta\left(\frac{d(x,x_0)}{\delta^{-1}}\right)=\eta\left(\frac{d_g(x,x_0)}{\delta^{-1}}\right).
\end{align}
Then we obviously have that
\begin{gather}\label{cutoff}
\phi_\delta\in C^{0,1}\big(B(x_0,1/\delta)\cap\mathcal{R}_X\big),\\\nonumber
|\nabla \phi_\delta|\le 3\delta \quad \text{ on }\quad \mathcal{R}_X,
\\\nonumber
\phi_\delta(x)=1\quad\text{ on }\quad B(x_0,1/2\delta)\cap \mathcal{R}_X.
\end{gather}
We shall apply the logarithmic Sobolev inequality given by the local $\mu$-functional (Proposition \ref{prop: local-log-Sobolev-of-mu}) to the test function $\eta_r^2\phi_\delta^2(4\pi)^{-\frac{n}{2}}e^{-f_0}\in C_{c}^{0,1}\big(B(x_0,1/\delta)\cap\mathcal{R}_X\big)$, where $\eta_r$ is the cut-off function given by Proposition \ref{prop:regular cutoff}. Let us first of all normalize its integral. If $\delta^{-1}>2\sqrt{2H_n}$, then we have
\begin{align}\label{lower bound of V}
     V
    &:= \int_{\mathcal{R}_X} \eta_r^2\phi_\delta^2(4\pi)^{-\frac{n}{2}}
     e^{- f_0}\, dg 
    \\\nonumber
    &=\int_{\mathcal{R}_X}\phi_\delta^2 (4\pi)^{-\frac{n}{2}}
     e^{- f_0}\, dg- \int_{\mathcal{R}_X} \phi_\delta^2(1-\eta_r^2)(4\pi)^{-\frac{n}{2}}
     e^{-f_0}\, dg
    \\\nonumber
    &\ge \int_{  B\big(x_0,\frac{1}{2}\delta^{-1}\big)}\, d\nu -\int_{\{0\le\eta_r<1\}\cap\mathcal{R}_X\cap B(x_0,\delta^{-1})} (4\pi)^{-\frac{n}{2}}
     e^{-f_0}\, dg
    \\\nonumber
    &\ge \nu\left( B\left(x_0,\tfrac{1}{2}\delta^{-1}\right)\right)-C(n)e^{-W}\int_{\{0\le\eta_r<1\}\cap\mathcal{R}_X\cap B(x_0,\delta^{-1})} \, dg
    \\\nonumber
    &\ge 1- 4H_n\delta^2-C(\delta,W,\sigma)r^{4-\sigma},
\end{align}
where we have applied Lemma \ref{lm: Hn-center-on-metric-soliton} (in combination with \cite[Proposition 3.13]{Bam20a}) and Theorem \ref{thm:potential-function}.
Define
\[
     u := \phi_\delta\eta_r\sqrt{(4\pi)^{-\frac{n}{2}}
     e^{-f_0}/V},
\]
and we shall use it as the test function for the local $\mu$-functional. Note that
\[
\int_{\mathcal{R}_X}u^2\,dg=1\qquad \text{ and }\qquad u\in C^{0,1}_c(\mathcal{R}_X).
\]
By Proposition \ref{prop: local-log-Sobolev-of-mu}, we have

\begin{align}\label{last but two}
     & \mu\big(B(x_0,\delta^{-1})\cap\mathcal{R}_X,g,1\big) 
     \\\nonumber
    &\le \int_{\mathcal{R}_X} \left(4|\nabla u|^2
        + Ru^2
    \right)\, dg-\int_{\mathcal{R}_X} u^2\log u^2\,dg
    - \tfrac{n}{2}\log (4\pi) - n\\\nonumber
    &= 
    \int_{\mathcal{R}_X} \left(4\left|\nabla \log (\phi_\delta\eta_r)-\tfrac{1}{2}\nabla f_0\right|^2u^2
        + Ru^2\right)\,dg
        \\\nonumber
        &\qquad -\int_{\mathcal{R}_X}\left(\log (\phi_\delta\eta_r)^2-\tfrac{n}{2}\log 4\pi - f_0-\log V\right)u^2\,dg - \tfrac{n}{2}\log (4\pi) - n
        \\\nonumber
&=\frac{1}{V}\int_{\mathcal{R}_X}\left(4\left|\nabla(\phi_\delta\eta_r)\right|^2-4(\phi_\delta\eta_r)\left\langle\nabla(\phi_\delta\eta_r),\nabla f_0\right\rangle\right)(4\pi)^{-\frac{n}{2}}e^{- f_0}\,dg
\\\nonumber
&\qquad+\int_{\mathcal{R}_X}\left(\left|\nabla f_0\right|^2+R+ f_0-n\right)u^2\,dg+\log V
-\frac{1}{V}\int_{\mathcal{R}_X}\left((\phi_\delta\eta_r)^2\log(\phi_\delta\eta_r)^2\right) (4\pi)^{-\frac{n}{2}}e^{- f_0}\,dg       
   \\\nonumber
   &=\frac{1}{V}\int_{\mathcal{R}_X}\left(4\left|\nabla(\phi_\delta\eta_r)\right|^2+2(\phi_\delta\eta_r)^2\Delta f_0-2(\phi_\delta\eta_r)^2\left|\nabla f_0\right|^2\right)(4\pi)^{-\frac{n}{2}}e^{- f_0}\,dg
   \\\nonumber
&\qquad+\int_{\mathcal{R}_X}\left(\left|\nabla f_0\right|^2+R+ f_0-n\right)u^2\,dg+\log V
-\frac{1}{V}\int_{\mathcal{R}_X}\left((\phi_\delta\eta_r)^2\log(\phi_\delta\eta_r)^2\right) (4\pi)^{-\frac{n}{2}}e^{- f_0}\,dg       
   \\\nonumber
&=\int_{\mathcal{R}_X}\left(2\Delta  f_0-\left|\nabla f_0\right|^2+R+ f_0-n\right)u^2\,dg+\log V
\\\nonumber
&\qquad +\frac{4}{V}\int_{\mathcal{R}_X}|\nabla (\phi_\delta\eta_r)|^2(4\pi)^{-\frac{n}{2}}e^{- f_0}\,dg
-\frac{1}{V}\int_{\mathcal{R}_X}\left((\phi\eta_r)^2\log(\phi\eta_r)^2\right) (4\pi)^{-\frac{n}{2}}e^{- f_0}\,dg 
\\\nonumber
&=W+\log V+\frac{4}{V}\int_{\mathcal{R}_X}|\nabla (\phi_\delta\eta_r)|^2(4\pi)^{-\frac{n}{2}}e^{- f_0}\,dg
-\frac{1}{V}\int_{\mathcal{R}_X}\left((\phi_\delta\eta_r)^2\log(\phi_\delta\eta_r)^2\right) (4\pi)^{-\frac{n}{2}}e^{- f_0}\,dg. 
\end{align}
We estimate the last two terms in the above formula. First of all, by \eqref{cutoff} and Proposition  \ref{prop:regular cutoff}, we have
\begin{align}\label{last-but-one}
    &\ \frac{4}{V}\int_{\mathcal{R}_X}|\nabla (\phi_\delta\eta_r)|^2(4\pi)^{-\frac{n}{2}}e^{- f_0}\,dg
    \\\nonumber
    \le&\ \frac{8}{V}\int_{\mathcal{R}_X}\left(\eta_r^2|\nabla\phi_\delta|^2+\phi_\delta^2|\nabla\eta_r|^2\right)(4\pi)^{-\frac{n}{2}}e^{- f_0}\,dg
    \\\nonumber
    \le &\ \frac{8}{V}\sup_{\mathcal{R}_X}|\nabla\phi_\delta|^2 \cdot\nu(X)+\frac{8}{V}\int_{B(x_0,\delta^{-1})\cap\mathcal{R}_X}|\nabla\eta_r|^2\,dg
    \\\nonumber
    \le &\ \frac{72\delta^2}{V}+\frac{C(n)}{V}\int_{\{|\nabla\eta_r|\neq 0\}\cap B(x_0,\delta^{-1})\cap\mathcal{R}_X}r^{-2}\,dg
    \\\nonumber
    \le & \frac{1}{V}\left(72\delta^2+C(\delta,W,\sigma)r^{2-\sigma}\right).
\end{align}
On the other hand, by Jensen's inequality, we have
\begin{align}\label{last}
 \frac{1}{V}\int_M\left((\phi_\delta\eta_r)^2\log(\phi_\delta\eta_r)^2\right) (4\pi)^{-\frac{n}{2}}e^{- f_0}\,dg
& = \frac{1}{V}\int_M(\phi_\delta\eta_r)^2\log(\phi_\delta\eta_r)^2 d\nu
\\\nonumber
&\ge \frac{1}{V}\left(\int_M(\phi_\delta\eta_r)^2\,d\nu\right)\log\left(\int_M(\phi_\delta\eta_r)^2\,d\nu\right)
\\\nonumber
&=\frac{1}{V}\cdot V\log V
\\\nonumber
&=\log V.
\end{align}

By fixing $\sigma=1/2$, $\delta\le \overline{\delta}(n)$, and $r\le\overline{r}(\delta,W,\sigma)$, we have by \eqref{lower bound of V} that
$$V\ge 1-C_n\delta^2\ge \tfrac{1}{2}.$$
Thus, combining \eqref{last but two}, \eqref{last-but-one}, and \eqref{last}, we obtain
\begin{align}\label{eq:mu-functional-upper}
    W\ge \mu\big(B(x_0,\delta^{-1})\cap\mathcal{R}_X,g,1\big)-\delta-C_n\delta^2.
\end{align}
Finally, letting $\delta=\delta(n)$ be small enough, we then have that if
\begin{align*}
    \mu\big(B(x_0,\delta^{-1})\cap\mathcal{R}_X,g,1\big)\ge -\delta,
\end{align*}
then
\begin{align*}
    W\ge -\varepsilon(n),
\end{align*}
where $\varepsilon(n)$ is the constant in Theorem \ref{thm-Nash-gap}; this finishes the proof of Theorem \ref{thm: local-gap}.
\\

Taking $\delta\to 0$ in \eqref{eq:mu-functional-upper}, we also have the following estimate.
\begin{Corollary}\label{coro:mu-functional-upper}
    On a metric soliton, we have
    \begin{align*}
        W\ge \mu(\mathcal{R}_X,g,1).
    \end{align*}
\end{Corollary}

\section{Global \texorpdfstring{$\nu$}{nu}-functional and Sobolev inequality}

In this section, we shall prove Theorem \ref{thm: nu-functional}. In view of Corollary \ref{coro:mu-functional-upper}, we need only to show the lower estimate of the $\nu$-functional; the whole section is devoted to the proof of the following proposition.

\begin{Proposition}\label{prop:global-sobolev-1}
    Let $(X,d,\nu)$ be the model of an $n$-dimensional noncollapsed $\mathbb{F}$-limit metric soliton and $(\mathcal{R}_X,g,f_0)$ its regular part. Let $W\in(-\infty,0]$ be the soliton entropy. Then we have
    $$\mu(\mathcal{R}_X,g,\tau)\ge W\quad\text{ for all }\quad \tau>0.$$
    In particular, we have    $$\nu(\mathcal{R}_X,g)\ge W$$
    where $\mu$ and $\nu$ are Perelman's functionals defined in Subsection 6.1.
\end{Proposition}

Let us consider a metric soliton $(\mathcal{X},(\nu_t)_{t<0})$ in the statement of the proposition. By Theorem \ref{thm: metric soliton is ancient} and the reasoning at the beginning of Section 5, we may pick a sequence $\{(M^i,g^i_t,x_i)\}_{i=1}^\infty$ satisfying $T_i\nearrow +\infty$, \eqref{eq:nash-entropy-converge},  \eqref{eq: limiting-definition-refined}, and \eqref{Sc-almost-positivity}. By \cite[Theorem 9.31]{Bam20b} and \cite[Theorem 2.5]{Bam20c}, there is an open-set-exhaustion $U_1\subset U_2\subset\hdots \subset U_i\subset \hdots\subset \mathcal{R}\subset\mathcal{X}$ and a sequence of smooth time-preserving diffeomorphisms $\psi_i$ 
\begin{align*}
    \psi_i:U_i\longrightarrow V_i\subset M^i\times [-T_i,0],
\end{align*}
and $\varepsilon_i\to 0$ such that \cite[Theorem 9.31(a)---(f)]{Bam20b} hold. In particular, we have
\begin{align}\label{eq:convergence-of-metric}
    \left\|\psi_i^*g^i_t-g_t\right\|_{C^{[\varepsilon_i^{-1}]}(U_i)}\le \varepsilon_i.
\end{align}
Let $x_0=x_0(-1)\in\mathcal{R}_{-1}$ be a center of the metric soliton, and let $$(z_i,-1)=\psi_i(x_0)\in M^i\times\{-1\}$$ be the images of $x_0$. Arguing in the same way as the proof of \eqref{near-the-Hn-center-1}, we have that whenever $i$ is large enough, each $(z_i,-1)$ is at most $3\sqrt{H_n}$-away from an $H_n$-center of $(x_i,0)$. Thus, in view of \cite[Definition 3.10]{Bam20a}, we have
\begin{align}\label{eq:near-the-Hn-center-3}
    d^{g^i_{-1}}_{W_1}\left(\delta_{z_i},\nu_{x_i,0\,|\,-1}\right)\le 4\sqrt{H_n}.
\end{align}

By the nature of local smooth convergence, the geometry of $\mathcal{R}_X$ near a center is close to that of $(M^i,g^i_{-1})$ near an $H_n$-center of $(x_i,0)$. In spite of the formation of singularity along the converging sequence, the Sobolev inequality is actually unaffected. The idea is to send functions compactly supported on $\mathcal{R}_X (\cong\mathcal{R}_{-1})$ to functions on $(M^i,g^i_{-1})$, and apply the Sobolev inequality therein. However, since $(M^i,g^i_t)$ is not necessarily ancient, we cannot directly apply the uniform Sobolev inequality \cite{CMZ23}, but the local Sobolev inequality \cite{CMZ23a} is sufficient for our purpose.

\begin{Proposition}[Local Sobolev inequality {\cite[Theorem 1.1]{CMZ23a}}]\label{prop:local Sobolev}
    For any $\varepsilon>0$ and for each $i$, we have
    \begin{align*}
        \nu\left(B_{g^i_{-1}}\left(z_i,\varepsilon\sqrt{T_i/2}\right),g^i_{-1},\varepsilon T_i/2\right)\ge \mathcal{N}_{z_i,-1}(T_i/2)-\sqrt{n}\varepsilon-\tfrac{n}{2}\varepsilon-\tfrac{n}{2}\log(1+\varepsilon).
    \end{align*}
\end{Proposition}

For any fixed positive scale $\tau>0$, let us consider an arbitrary test function $u\in C^{\infty}_c(\mathcal{R}_X)$ satisfying 
\begin{align}
 \int_{\mathcal{R}_X}u^2\,dg=1.
\end{align} 
Write $K=\operatorname{\spt} u$ and fix an  $A>0$ such that
$$K\Subset \mathcal{R}_X\cap B(x_0,A)\subset \mathcal{R}_{-1}.$$
Obviously, whenever $i$ is large enough, we have
\begin{align*}
    K\Subset U_i\quad \text{ and }\quad \psi_{i,-1}(K)\subset B_{g^i_{-1}}(z_i,2A).
\end{align*}
We may then consider the pull-backs of $u$ via $\psi_{i,-1}$. Define
\begin{align*}
    u_i:=u\circ\psi_{i,-1}^{-1}: M^i\to\mathbb{R}.
\end{align*}
The following lemma is a straightforward consequence of \eqref{eq:convergence-of-metric}.
\begin{Lemma}
    Whenever $i$ is large enough, we have
    \begin{enumerate}[(1)]
        \item $u_i$ is compactly supported in $B_{g^i_{-1}}(z_i,2A)$.
        \item Let $\displaystyle V_i^2:=\int_{M^i} u_i^2\,dg^i_{-1}$, then $V_i=1+\Psi(\varepsilon_i)$.
        \item 
        \begin{align}
            \int_{M^i}u_i^2\log u_i^2\,dg^i_{-1}&\ =(1+\Psi(\varepsilon_i))\int_{\mathcal{R}_X}u^2\log u^2\,dg,\label{eq:almost-mu-functional-test-1}
            \\
            \int_{M^i}\tau(4|\nabla_{g^i_{-1}}u_i|^2+R_{g^i_{-1}}u_i^2)\,dg^i_{-1}&\ =(1+\Psi(\varepsilon_i))\int_{\mathcal{R}_X}\tau(4|\nabla u|^2+Ru^2)\,dg.\label{eq:almost-mu-functional-test-2}
        \end{align}
    \end{enumerate}
    Here (and in the following) $\Psi(\varepsilon)$ stands for a general constant that converges to zero as $\varepsilon\to 0$, which depends on $\tau$ and $u$ and varies from line to line; $\varepsilon_i$ is the constant in \eqref{eq:convergence-of-metric}.
\end{Lemma}

Taking $u_i/V_i$ as the test function of $\mu\big(B_{g^i_{-1}}(z_i,2A),g^i_{-1},\tau\big)$, we have
\begin{Lemma}\label{lm:almost-test-function}
    \begin{align*}
        &\ \int_{\mathcal{R}_X}\left(\tau(4|\nabla u|^2+Ru^2)-u^2\log u^2\right)\,dg-\left(\frac{n}{2}\log(4\pi\tau)+n\right)
        \\
    \ge &\ (1+\Psi(\varepsilon_i))\mu\Big(B_{g^i_{-1}}(z_i,2A),g^i_{-1},(1+\Psi(\varepsilon_i))\tau\Big)+\Psi(\varepsilon_i).
    \end{align*}
\end{Lemma}
\begin{proof}
Let $v_i=u_i/V_i$, by \eqref{eq:almost-mu-functional-test-1} we have
    \begin{align*}
        \int_{M^i}v_i^2\log v_i^2\,dg^i_{-1}&\ =\frac{1}{V_i^2}\int_{M^i}u_i^2\log u_i^2\,dg^i_{-1}-\frac{\log V_i^2}{V_i^2}\int_{M^i} u_i^2\,dg^i_{-1}
        \\
        &\ =(1+\Psi(\varepsilon_i))\int_{\mathcal{R}_X}u^2\log u^2\,dg+\Psi(\varepsilon_i),
    \end{align*}
    or, equivalently
    \begin{align}\label{eq:almost-mu-functional-test-3}
        \int_{\mathcal{R}_X}u^2\log u^2\,dg =(1+\Psi(\varepsilon_i))\int_{M^i}v_i^2\log v_i^2\,dg^i_{-1}+\Psi(\varepsilon_i).
    \end{align}
    
    On the other hand, \eqref{eq:almost-mu-functional-test-2} implies that
    \begin{align}\label{eq:almost-mu-functional-test-4}
        \int_{\mathcal{R}_X}\tau(4|\nabla u|^2+Ru^2)\,dg=(1+\Psi(\varepsilon_i))\int_{M^i}\tau(4|\nabla_{g^i_{-1}}v_i|^2+R_{g^i_{-1}}v_i^2)\,dg^i_{-1}.
    \end{align}

    The lemma follows from combining \eqref{eq:almost-mu-functional-test-3} and \eqref{eq:almost-mu-functional-test-4}.
    
\end{proof}

It turns out that one needs only to bound the local $\mu$-functional at $(z_i,-1)$, which in turn depends on the Nash entropy at $(z_i,-1)$ by Proposition \ref{prop:local Sobolev}. The latter can be easily estimated via \cite[Corollary 5.11]{Bam20a}.

\begin{Lemma}\label{lm:almost-converging-nash-entropy}
    For any $\varepsilon>0$, we have $\mathcal{N}_{z_i,-1}(T_i/2)\ge W-\varepsilon$ whenever $i$ is large enough.
\end{Lemma}
\begin{proof}
    Let us apply \cite[Corollary 5.11]{Bam20a} with
    \begin{align*}
        (x_1,t_1)=(x_i,0),\quad (x_2,t_2)=(z_i,-1), \quad t^*=-1,\quad s=-1-T_i/2,\quad R_{\min}=-\tfrac{1}{T_i}.
    \end{align*}
    Note that the last term above is due to \eqref{Sc-almost-positivity}, and that we have the closeness between $(x_1,t_1)$ and $(x_2,t_2)$ \eqref{eq:near-the-Hn-center-3}. Thus,
    \begin{align*}
        \mathcal{N}_{z_i,-1}(T_i/2)\ge \mathcal{N}_{x_i,0}(T_i/2+1)-\frac{C(n)}{T_i}.
    \end{align*}
    The lemma follows from the above inequality and the convergence of the Nash entropy \eqref{eq:nash-entropy-converge}.
\end{proof}

Finally, let us fix any positive number $\varepsilon>0$ and take $i$ large enough such that $\varepsilon T_i/2\gg \tau$. Proposition \ref{prop:local Sobolev}, Lemma \ref{lm:almost-test-function}, and Lemma \ref{lm:almost-converging-nash-entropy} imply that
\begin{align*}
    \int_{\mathcal{R}_X}\left(\tau(4|\nabla u|^2+Ru^2)-u^2\log u^2\right)\,dg-\left(\frac{n}{2}\log(4\pi\tau)+n\right)\ge (1+\Psi(\varepsilon_i))W-\Psi(\varepsilon_i)-\Psi(\varepsilon).
\end{align*}
First taking $i\to\infty$ and then $\varepsilon\to 0$, we have
$$\int_{\mathcal{R}_X}\left(\tau(4|\nabla u|^2+Ru^2)-u^2\log u^2\right)\,dg-\left(\frac{n}{2}\log(4\pi\tau)+n\right)\ge W.$$
Proposition \ref{prop:global-sobolev-1} follows since the test function $u$ is arbitrarily chosen. As a consequence, we have proved Theorem \ref{thm: nu-functional}.

\section{Volume growth lower estimate}


This section is devoted to the proof of Theorem \ref{thm:linear-volume-lower-bound}.  The idea of the proof, according to \cite{MW12,LW20}, consists of two steps. The first step is to show that a noncompact (in the sense of infinite diameter) metric soliton has infinite volume. In this step, we apply the Sobolev inequality (Corollary \ref{coro: Sobolev}) to show that the volume of unit balls does not decay at a rate higher than $r^{-n}$, where $r$ is the distance from the origin. While an argument similar to \cite{MW12} shows that a metric soliton with finite volume must have exponential decay rate for the volume of unit balls. The second step is to apply the argument of \cite{LW20}, and show that sublinear volume growth contradicts the infinity of the volume.


As always, we consider an $n$-dimensional noncollapsed $\mathbb{F}$-limit metric soliton. Let $(X,d,\nu)$ be its model and $(\mathcal{R}_X,g,f_0)$ the regular part. Let $x_0\in\mathcal{R}_X$ be the center of the soliton and $W$ the soliton entropy.

\subsection{Preliminary volume decay lower bound}



\begin{Lemma}\label{lm: local-volume-lower-bound}
    For any $x\in X$, if $R\le r^{-2}$ on $B(x,r)\cap\mathcal{R}_X$, then
    \begin{align*}
        |B(x,r)\cap\mathcal{R}_X|\ge c(n)e^{W}r^n
    \end{align*}
\end{Lemma}
\begin{proof}
    This is a straightforward consequence of Corollary \ref{coro: Sobolev} and an iteration argument. First of all, we consider the case $x\in\mathcal{R}_X$. Fix $\sigma=1/2$, $A\gg r+d_g(x_0,x)$, and any $\delta>0$ (corresponding to $r$ in Proposition \ref{prop:regular cutoff}), let $\eta_\delta$ be the cut-off function given by Proposition \ref{prop:regular cutoff}. Letting
    \begin{align*}
        u:=\left(r-d_g(x,\cdot)\right)_+: \ \mathcal{R}_X\to \mathbb{R}
    \end{align*}
    (note that $u$ is $1$-Lipschitz on $\mathcal{R}_X$) and applying the Sobolev inequality \eqref{eq: the-uniform-Sobolev} to $u\eta_\delta$, we have
    \begin{align*}
        \left(\int_{\mathcal{R}_X}(u\eta_\delta)^{\frac{2n}{n-2}}\,dg\right)^{\frac{n-2}{2}}\le&\  C(n)e^{-\frac{2W}{n}}\int_{\mathcal{R}_X}\left(4|\nabla(u\eta_\delta)|^2+R(u\eta_\delta)^2\right)\,dg
        \\
        \le&\ C(n)e^{-\frac{2W}{n}}\int_{\mathcal{R}_X}\left(8|\nabla u|^2\eta_\delta^2+8u^2|\nabla\eta_{\delta}|^2+\eta_\delta^2\right)\,dg
        \\
        \le &\ C(n)e^{-\frac{2W}{n}}\left(9|B(x,r)\cap\mathcal{R}_X|+C(A,W,\sigma)\delta^{2-\sigma}\right).
    \end{align*}
    Taking $\delta\to 0$ and applying Proposition \ref{prop:regular cutoff} as well as the dominated convergence theorem, we have
    \begin{align*}
        \left(\tfrac{r}{2}|B(x,r/2)\cap\mathcal{R}_X|\right)^{\frac{n-2}{n}}\le  C(n)e^{-\frac{2W}{n}}|B(x,r)\cap\mathcal{R}_X|.
    \end{align*}
    Since $x\in\mathcal{R}_X$, the geometry in a very small disk around $x$ is almost-Euclidean. Thus, iterating the inequality above, we obtain
    $$|B(x,r)\cap\mathcal{R}_X|\ge c(n)e^Wr^n$$
    as in the classical case. 
    
    Finally, if $x\in X\setminus \mathcal{R}_X$, we can then apply the former result to a point $x'\in \mathcal{R}_X\cap B(x,r/100)$, and the conclusion follows in like manner.
\end{proof}

\begin{Lemma}\label{lm:polynomial-volume-decay}
    If $d(x,x_0)\ge 1$, then
    $$|B(x,1)\cap \mathcal{R}_X|\ge c(n)|B(x_0,1)\cap \mathcal{R}_X|\cdot\frac{1}{d^n(x_0,x)}.$$
\end{Lemma}

\begin{proof}
    The third equation in \eqref{eq: basic-properties-2} and Theorem \ref{thm:potential-function} implies that
    \begin{align}\label{Rupb}
        R(x)\le \frac{1}{4}d^2(x_0,x)+C(n)\quad \text{ for all }\quad x\in \mathcal{R}_X.
    \end{align}
    Thus, if we can show 
    \begin{align*}
        e^W\ge c(n)|B(x_0,1)\cap \mathcal{R}_X|,
    \end{align*}
    the lemma follows from Lemma \ref{lm: local-volume-lower-bound}; but this is simply Corollary \ref{volume-upper-1}.
\end{proof}

\subsection{Differential equation of the volume and the scalar curvature}

Similar to \cite{CZ10}, let $$\rho:=2\sqrt{f_0-W}\ge 0,$$and we shall consider the volumes of the sub-level-sets of $\rho$ and the integrals of the scalar curvature therein. It follows from Theorem \ref{thm:potential-function} that
\begin{equation}\label{rhoest}
2d_g(x,x_0)/3-c_n\leq \rho(x) \leq d_g(x,x_0)+c_n,
\end{equation}
where $c_n>0$ is a dimensional constant. We also define
\begin{eqnarray*}
D(s)&:=&\{x\in X:\,\rho(x)\leq s \}\\
V(s)&:=&\int_{D(s)\cap \mathcal{R}_X}\,dg\\
\chi(s)&:=&\int_{D(s)\cap \mathcal{R}_X} R\,dg.
\end{eqnarray*}
In spite of the existence of singular points, $V(s)$ and $\chi(s)$ are still nice in the measure-theoretic sense.
\begin{Lemma}\label{lm:absolutely-continuous-V-and-chi}
   $V(s)$ and $\chi(s)$ are absolutely continuous, and 
    \begin{align}\label{eq: ODE-for-V-and-chi}
        V'(s)&\ = \int_{\{\rho=s\}\cap\mathcal{R}_X}\frac{1}{|\nabla\rho|}\,dA= \frac{s}{2}\int_{\{\rho=s\}\cap\mathcal{R}_X}\frac{1}{|\nabla f_0|}\,dA,
        \\\nonumber
        \chi'(s)&\ = \int_{\{\rho=s\}\cap\mathcal{R}_X}\frac{R}{|\nabla\rho|}\,dA= \frac{s}{2}\int_{\{\rho=s\}\cap\mathcal{R}_X}\frac{R}{|\nabla f_0|}\,dA,
    \end{align}
    for almost every $s\ge 0$.
\end{Lemma}

\begin{proof}

Thanks to Lemma \ref{lm: R-integral-const},  Theorem \ref{thm:potential-function}, and Corollary \ref{volume-upper-1}, both $V$ and $\chi$ are nondecreasing finite real-valued functions in $s$.
In view of the regularity result \cite[Corolary 1.3]{Kot13}, both $g$ and $\nabla f_0$ are real analytic in the geodesic normal coordinates. Consequently, both the soliton potential $f_0=W+R+|\nabla f_0|^2$ and the scalar curvature $R$ are real analytic function on $\mathcal{R}_X$. Applying \cite{M20} to the analytic function $|\nabla f_0|^2=f_0-R-W$, we have that the singular set $\mathcal{C}:=\{x\in \mathcal{R}_X: |\nabla \rho|(x)=0 \}$ has zero measure. Hence, $\frac{1}{|\nabla \rho|}$ is a nonnegative measurable function on $\mathcal{R}_X\setminus\mathcal{C}$. Furthermore, by the Morse-Sard Theorem \cite{SS72}, the singular value set $\rho(\mathcal{C})$ is at most countable.

By applying the coarea formula to $\rho$ on $\mathcal{R}_X\setminus\mathcal{C}$, we also have
\begin{eqnarray*}
V(s)&=& |D(s)\cap \mathcal{R}_X\cap \mathcal{C}|+|D(s)\cap \mathcal{R}_X\setminus\mathcal{C}|= |D(s)\cap \mathcal{R}_X\setminus\mathcal{C}|\\
&=&\int_{D(s)\cap \mathcal{R}_X\setminus\mathcal{C}}\,dg =\int_0^s\int_{\{x\in \mathcal{R}_X\setminus\mathcal{C}:\, \rho(x)=t\}}\frac{1}{|\nabla \rho|}\,dA\, dt.
\end{eqnarray*}
Similarly, 
\[
\chi(s)=\int_{D(s)\cap \mathcal{R}_X}R\,dg=\int_{D(s)\cap \mathcal{R}_X\setminus\mathcal{C}}R\,dg=
\int_0^s\int_{\{x\in \mathcal{R}_X\setminus\mathcal{C}:\, \rho(x)=t\}}\frac{R}{|\nabla \rho|}\,dA\, dt.
\]
Hence the nonnegative functions
\begin{eqnarray*}
t&\mapsto& \int_{\{x\in \mathcal{R}_X\setminus\mathcal{C}:\, \rho(x)=t\}}\frac{1}{|\nabla \rho|}\,dA\\
t&\mapsto& \int_{\{x\in \mathcal{R}_X\setminus\mathcal{C}:\, \rho(x)=t\}}\frac{R}{|\nabla \rho|}\,dA
\end{eqnarray*}
are integrable on $[0,s]$, and $V$ and $\chi$ are absolutely continous functions with
\begin{eqnarray*}
V'(s)&=&\int_{\{x\in \mathcal{R}_X\setminus\mathcal{C}:\, \rho(x)=s\}}\frac{1}{|\nabla \rho|}\,dA,
\\
\chi'(s)&=&\int_{\{x\in \mathcal{R}_X\setminus\mathcal{C}:\, \rho(x)=s\}}\frac{R}{|\nabla \rho|}\,dA
\end{eqnarray*}
for almost every $s$. Since for each regular value $s$ of $\rho$, it holds that
\[
\{x\in \mathcal{R}_X\setminus\mathcal{C}:\, \rho(x)=s\}=\{x\in \mathcal{R}_X:\, \rho(x)=s\},
\]
so the integral on both sets are equal. By The Morse-Sard Theorem again, the singular value consists of a null set;  \eqref{eq: ODE-for-V-and-chi} follows immediately.

\end{proof}

We then show a statement analogous to \cite[Lemma 3.1]{CZ10}.
\begin{Lemma}\cite{{CZ10}}\label{CZlem} For almost every positive number $s$, 
\begin{equation} \label{diffeq}
    \frac{n}{2}V(s)-\frac{s}{2} V'(s)=\chi(s)-\frac{2}{s}\chi'(s),
\end{equation}
and 
\begin{equation} \label{intbddR}
\chi(s) \leq  \frac{n}{2}V(s).
\end{equation}
\end{Lemma}
\begin{proof} 
Let $s$ be a positive regular value of $f_0$ at which both $V$ and $\chi$ are differentiable.  By \eqref{rhoest}
\[
D(s)\subseteq B_g(x_0, 3(s+c_n)/2).
\]
We apply Proposition \ref{prop:regular cutoff} with $A:=3(s+c_n)+\underline{A}(n)$ and $\sigma<2$ to get a cutoff function $\eta_r$ for each $r\le \overline{r}(A,W,\sigma)$; here $0<\sigma<2$ is arbitrarily fixed.
By taking the trace of the soliton equation \eqref{eq: basic-properties-2}, we have
\[
\Delta f_0+R=\frac{n}{2}.
\]
We multiply the equation by $\eta_r^2$ and integrate it over $D(s)\cap \mathcal{R}_X$:
\begin{equation}\label{rto0}
    \int_{D(s)\cap \mathcal{R}_X}\eta_r^2 \Delta f_0\, dg+ \int_{D(s)\cap \mathcal{R}_X} \eta^2_r R\, dg=\frac{n}{2}\int_{D(s)\cap \mathcal{R}_X} \eta^2_r \, dg.
\end{equation}
We shall estimate the equation term by term. 

First of all, we have
\[
0\le\frac{n}{2}\int_{D(s)\cap \mathcal{R}_X} (1-\eta^2_r) \, dg\leq \frac{n}{2}\int_{\mathcal{R}_X\cap B(x_0,A)\cap\{0\le \eta_r<1\}}\,dg \le C(A,W,\sigma)r^{4-\sigma}\to 0\ \text{  as  }\ r\to 0.
\]
Hence 
\[
\frac{n}{2}\int_{D(s)\cap \mathcal{R}_X} \eta^2_r \, dg\to \frac{n}{2}V(s)\ \text{  as  }\ r\to 0.
\]

Using also the third equation of \eqref{eq: basic-properties-2}, it holds that $R\le \frac{1}{4}\rho^2$. Thus
\[
0\le \int_{D(s)\cap \mathcal{R}_X} (1-\eta^2_r)R \, dg\leq  \frac{s^2}{4}\int_{D(s)\cap \mathcal{R}_X} (1-\eta^2_r) \, dg\leq C(A,W,\sigma) s^2 r^{4-\sigma}\to 0
\]
and 
\[
\int_{D(s)\cap \mathcal{R}_X} \eta^2_rR \, dg\to \int_{D(s)\cap \mathcal{R}_X} R \, dg\  \text{  as  }\  r\to 0.
\]

 By Proposition \ref{prop:regular cutoff}, when applying the divergence theorem \cite[Theorem 16.32, Theorem 16.48]{Lee13} to $\eta_r^2\nabla f_0$ on $D(s)\cap \mathcal{R}_X$, the boundary integral occurs only on $\{x\in \mathcal{R}_X: \rho(x)=s\}$, which is a smooth hypersurface (recall that $s$ is a regular value of $f_0$). Thus
\begin{eqnarray*}
     \int_{D(s)\cap \mathcal{R}_X}\eta_r^2 \Delta f_0\, dg&=&- 2\int_{D(s)\cap \mathcal{R}_X}\eta_r\langle\nabla \eta_r, \nabla f_0\rangle\, dg+\int_{\{x\in \mathcal{R}_X: \rho(x)=s\}}\eta_r^2\frac{\langle\nabla_0 f, \nabla \rho\rangle}{|\nabla\rho|} \,dA\\
&=&- 2\int_{D(s)\cap \mathcal{R}_X}\eta_r\langle\nabla \eta_r, \nabla f_0\rangle\, dg+\int_{\{x\in \mathcal{R}_X: \rho(x)=s\}}\eta_r^2|\nabla f_0| \,dA
\end{eqnarray*}
\begin{eqnarray*}
2\int_{D(s)\cap \mathcal{R}_X}\eta_r\left|\langle\nabla \eta_r, \nabla f_0\rangle\right|\, dg&\le& C_0sr^{-1} \int_{\mathcal{R}_X\cap B(x_0,A)\cap\{|\nabla_g\eta_r|\neq0\}}dg\\
&\le& C(A,W,\sigma)sr^{3-\sigma}\to 0 \ \text{  as  }\  r\to 0.
\end{eqnarray*}
Moreover, Lemma \ref{lm:absolutely-continuous-V-and-chi} also implies that, for almost every $s\ge 0$,
\begin{align*}
\int_{\mathcal{R}_X\cap\{\rho=s\}}|\nabla f_0|\,dA\le \frac{s^2}{4}\int_{\mathcal{R}_X\cap\{\rho=s\}}\frac{1}{|\nabla f_0|} dA <+\infty.
\end{align*}
So, for such an $s$, Proposition \ref{prop:regular cutoff}(5) implies that
\begin{align*}
\lim_{r\to 0} \int_{\{x\in \mathcal{R}_X: \rho(x)=s\}}\eta_r^2|\nabla f_0| \,dA =\int_{\mathcal{R}_X\cap\{\rho=s\}}|\nabla f_0|\,dA
\end{align*}
due to the dominated convergence theorem.

Finally, \eqref{intbddR} follows from letting $r\to 0$ in \eqref{rto0} and the fact 
\begin{eqnarray*}
\frac{n}{2}V(s)- \int_{D(s)\cap \mathcal{R}_X} R \, dg =\int_{\{x\in \mathcal{R}_X: \rho(x)=s\}}|\nabla f_0| \,dA\ge 0.
\end{eqnarray*}
In like manner, \eqref{diffeq} also follows from taking $r\to 0$; note that the right-hand-side of \eqref{diffeq} comes from the following computation as in \cite{CZ10} using \eqref{eq: basic-properties-2}
\[
\int_{\{x\in \mathcal{R}_X: \rho(x)=s\}}|\nabla f_0| \,dA=\int_{\{x\in \mathcal{R}_X: \rho(x)=s\}}\frac{f_0-W-R}{|\nabla f_0|} \,dA=\frac{s}{2}V'(s)-\frac{2}{s}\chi'(s).
\]
\end{proof}

\subsection{Noncompact metric soliton has infinite volume}

Next we show that a noncompact metric soliton must have infinite volume. Taking $\tau=1$ in \eqref{eq: the-log-Sobolev}, we get the following logarithmic Sobolev inequality which shall be applied in this subsection. 
\begin{eqnarray}\label{logsob}
\W\int_{\mathcal{R}_X}u^2\,dg+\int_{\mathcal{R}_X}u^2\log u^2\,dg
\leq\int_{\mathcal{R}_X}\left(4|\nabla u|^2+Ru^2\right)\,dg+\int_{\mathcal{R}_X}u^2\,dg\cdot \log\left(\int_{\mathcal{R}_X}u^2\,dg\right),
\end{eqnarray}
for $u\in C^{0,1}_c(\mathcal{R}_X)$. Here we have defined
$$\W:= \min\left\{W+n+\frac{n}{2}\log (4\pi),0\right\}$$
for notational simplicity. The following theorem is prove by an argument similar to that of \cite{MW12}.

\begin{Proposition}\label{infvol}Suppose that $\mathcal{R_X}$ has infinite diameter. Then $\left|\mathcal{R_X}\right|=\infty$.
\end{Proposition}
\begin{proof} Suppose on the contrary that $\left|\mathcal{R_X}\right|<\infty$. We will show that the volume of the unit ball decays at least exponentially. This contradicts the polynomial decay in Lemma \ref{lm:polynomial-volume-decay}. Similar to \cite{MW12}, we consider the Lipschitz function
\begin{equation*}
    u_t(s)=\int^s_{-\infty}\chi_{[t-1,t]}(\tau)\,d\tau=\left\{
\begin{array}{rl}
1 &\text{  if  } t\le s;\\
s-t+1 &\text{  if  } t-1\le s\le t;\\
 0  &\text{  if  } s\le t-1.
\end{array} 
\right.
\end{equation*}
By the definition of $u_t$, for all $s\ge 0$
\[
|u_{t_1}(s)-u_{t_2}(s)|\le 4|t_1-t_2|.
\]
We write $u_t(x)=u_t(\rho(x))\in W^{1,2}(\mathcal{R}_X)$ and let 
\[
y(t)=\int_{\mathcal{R}_X} u_t^2\,dg.
\]
Since $u_{t_2}\le u_{t_1}$ if $t_1\le t_2$ and $\mathcal{R}_X$ has finite volume, $y(t)$ is a nonincreasing Lipschitz function in $t$. We claim that there exists positive constant $C=C(n,W)$ such that
\begin{equation}\label{diffineq}
    ty'(t)-2 y(t)\log y(t)\leq Cy(t-1).
\end{equation}

Assuming the claim at the moment, the same argument in \cite{MW12} shows that there are small positive constants $\varepsilon$ and $\delta$ such that 
\[
y(t)\leq \delta e^{-\varepsilon t}\quad \text{ for all $t\ge 1/\varepsilon$}.
\]
For any $t\ge 1/\varepsilon$, choose $x\in \mathcal{R}_X$ such that $\rho(x)=t+1$. Using \eqref{rhoest} and the fact that $|\nabla \rho|=\tfrac{|\nabla f_0|}{\sqrt{f_0-W}}\le 1$ on $B(x,1)\cap \mathcal{R}_X$,  for each $z\in B(x,1)\cap \mathcal{R}_X$, we may integrate $\nabla\rho$ along almost minimizing curve joining $x$ and $z$ to obtain 
\[
t\le \rho(z)\le t+2
\]
and 
\[
d_g(z,x_0)\leq c_n(t+1)\quad\text{ for all }\quad z\in B(x,1)\cap \mathcal{R}_X.
\]
Hence by Lemma \ref{lm:polynomial-volume-decay}
\[
c(n)|B(x_0,1)\cap \mathcal{R}_X|\cdot\frac{1}{(t+1)^n}\leq
|B(x,1)\cap \mathcal{R}_X|\leq y(t)\leq \delta e^{-\varepsilon t}, 
\]
which is impossible when $t$ is large; it remains to prove \eqref{diffineq} under the assumption $|\mathcal{R}_X|<+\infty$. 

Since \eqref{logsob} only holds for functions compactly supported on $\mathcal{R}_X$, we first verify that indeed \eqref{logsob} also holds when $u=u_t(x)$. Fix  $\sigma\in(0,1/2)$ and let $\delta$ be a small positive number, we define cutoff functions $\phi_\delta$ and $\eta_r$ in the same way as in Section 6.4, that is, $\phi_\delta$ is the cut-off function defined by \eqref{eq:cut-off-on-large-scale}, and $\eta_r$ is the cut-off function given by Proposition \ref{prop:regular cutoff} with $A=\delta^{-1}$ and $r\le \overline{r}(\delta,W,\sigma)$. We substitute $u=\phi_\delta\eta_r u_t(x)$ into \eqref{logsob}. Since
\[
\int_{\mathcal{R}_X}\phi_\delta^2(1-\eta_r^2) u_t^2(x)\,dg\leq \int_{\mathcal{R}_X\cap B(x_0,\delta^{-1})\cap\{0\le \eta_r<1\}}\,dg \le C(\delta,W,\sigma)r^{4-\sigma}\to 0\ \text{  as  }\ r\to 0,
\]
by first letting $r\to 0$ and then $\delta\to 0$, we have
\[
\int_{\mathcal{R}_X}\phi^2_\delta\eta_r^2 u_t^2(x)\,dg\to \int_{\mathcal{R}_X}u_t^2(x)\,dg=y(t),
\]
\[
\int_{\mathcal{R}_X}\phi_\delta^2\eta_r^2u_t^2(x)\,dg\log\left(\int_{\mathcal{R}_X}\phi_\delta^2\eta_r^2u_t^2(x)\,dg\right)\to y(t)\log y(t).
\]

By the similar computations as in \eqref{last but two}, \eqref{last-but-one} and \eqref{last}, we have
\begin{eqnarray*}
  &&  \int_{\mathcal{R}_X}\left(4\left|\nabla (\phi_\delta\eta_ru_t)\right|^2+R\phi_\delta^2\eta_r^2u_t^2\right)\,dg-  \int_{\mathcal{R}_X}\phi^2_\delta\eta_r^2u_t^2\log\left(\phi_\delta^2\eta_r^2u_t^2\right)\, dg\\
  &=& \int_{\mathcal{R}_X}\left(4\phi_\delta^2\eta_r^2\left|\nabla u_t\right|^2+R\phi_\delta^2\eta_r^2u_t^2\right)\,dg- \int_{\mathcal{R}_X}\phi_\delta^2\eta_r^2u_t^2\log (\phi_\delta^2\eta_r^2u_t^2)\, dg\\
  &&+\int_{\mathcal{R}_X}\left(4u_t^2\left|\nabla (\phi_\delta\eta_r)\right|^2+8\phi_\delta\eta_r u_t\langle\nabla (\phi_\delta\eta_r),\nabla u_t\rangle\right)\,dg.
\end{eqnarray*}
Since $0\le u_t(x)\le 1$ and $u_t(x)$ is $1$-Lipschitz in $x$,  $\left|\mathcal{R}_X\right|<\infty$,  and $-e^{-1}\le u_t^2\log u_t^2\leq 0$, we have
\begin{eqnarray*}
&& \bigg|\int_{\mathcal{R}_X}\phi_\delta^2\eta_r^2u_t^2\log\left(\phi_\delta^2\eta_r^2u_t^2\right)\, dg-\int_{\mathcal{R}_X}u_t^2\log u_t^2\, dg\bigg|\\
&=&\bigg|\int_{\mathcal{R}_X}\Big(u_t^2\phi_\delta^2\eta_r^2\log(\phi_\delta^2\eta_r^2) +\phi_\delta^2(\eta_r^2-1)u_t^2\log u_t^2+(\phi_\delta^2-1)u_t^2\log u_t^2\Big)\,dg\bigg|
\\
&\leq& \left|\int_{\{\phi_\delta>0\}\cap(\{\phi_\delta< 1\}\cup \{0<\eta_r<1\})}e^{-1}\,dg\right| + \left|\int_{\mathcal{R}_X}\phi_\delta^2(1-\eta_r^2)e^{-1}\, dg\right|+\left|\int_{\mathcal{R}_X}(1-\phi_\delta^2)e^{-1}\, dg\right|\\
&\leq& Cy(1/2\delta)+ C(\delta,W,\sigma) r^{4-\sigma}\to 0,
\end{eqnarray*}
\begin{eqnarray*}
\int_{\mathcal{R}_X}4u_t^2\left|\nabla (\phi_\delta\eta_r)\right|^2&\leq& \int_{\mathcal{R}_X}8u_t^2(\left|\nabla \phi_\delta\right|^2+\left|\nabla\eta_r\right|^2)\,dg\\
&\leq&  C(n)\left|\mathcal{R}_X\right|\delta^2+C(\delta,W,\sigma)r^{2-\sigma}\to 0,\\
\int_{\mathcal{R}_X}8\phi_\delta\eta_r u_t\left|\langle\nabla (\phi_\delta\eta_r),\nabla u_t\rangle\right|\,dg
&\leq& 
C\int_{\mathcal{R}_X}8\phi_\delta\eta_r u_t(\left|\nabla\phi_\delta\right|+\left|\nabla\eta_r\right|)\,dg\\
&\leq& C(n)\left|\mathcal{R}_X\right|\delta+C(\delta,W,\sigma)r^{3-\sigma}\to 0,
\end{eqnarray*}
the limits above are all obtained by first taking $r\to 0$ and then $\delta\to 0$. Hence from \eqref{logsob} and the definition of $u_t$, we have $$y(t)\leq C(\left|\mathcal{R}_X\right|-V(t-1))$$ and
\begin{align}\label{diffeq2}
(C\W-e^{-1})(\left|\mathcal{R}_X\right|-V(t-1))
&\leq \W y(t)-e^{-1}\int_{\mathcal{R}_X\setminus D(t-1)}\,dg\\\nonumber
&\leq  \W y(t)+\int_{\mathcal{R}_X}u_t^2\log u_t^2\,dg\\\nonumber
&\leq \int_{\mathcal{R}_X}(4|\nabla u_t|^2+Ru_t^2)\,dg+y(t)\log y(t)\\\nonumber
&\leq C(\left|\mathcal{R}_X\right|-V(t-1))+ \int_{\mathcal{R}_X}Ru_t^2\,dg+y(t)\log y(t).\nonumber
\end{align}

Moreover, as in \cite{MW12}, we have
\begin{align}\label{eq:theMW12ineq}
  \langle\nabla (u_t^2), \nabla f_0\rangle\leq \rho u_t\leq t u_t \text{  on  } \mathcal{R}_X\cap D(t)\setminus D(t-1)\text{ and vanishes elsewhere.} 
\end{align}
The following differential equation 
can also be verified straightforwardly.
\\

\noindent\textbf{Claim. }
\[
\frac{d}{dt}y(t)=-2\int_{\mathcal{R}_X\cap D(t)\setminus D(t-1)}u_t\,dg
\]
\emph{for every $t\ge 0$}.
\begin{proof}[Proof of the claim]
   For each $t\ge 0$ and $\varepsilon>0$, we may compute as follows.
\begin{align*}
    y(t+\varepsilon)-y(t)=&\ \int_{\mathcal{R}_X}\left(u_{t+\varepsilon}^2-u_t^2\right)\,dg
    \\
    =&\ -\int_{\mathcal{R}_X\cap D(t+\varepsilon-1)\setminus D(t-1)}\left(\rho-(t-1)\right)^2\,dg
    \\
    &\ -\int_{\mathcal{R}_X\cap D(t+\varepsilon)\setminus D(t)}\big(1-(\rho-(t+\varepsilon-1))^2\big)\,dg
    \\
    &\ -\varepsilon\int_{\mathcal{R}_X\cap D(t)\setminus D(t+\varepsilon-1)}\big(2(\rho-(t-1))-\varepsilon\big)\,dg.
\end{align*}
The first observation obtained from the above formula is that $y(t)$ is locally Lipschitz (and hence absolutely continuous), since each term on the right-hand-side can be bounded with $C(n)|\mathcal{R}_X|\varepsilon$. Secondly, we have
\begin{align*}
    &\frac{1}{\varepsilon}\left|\int_{\mathcal{R}_X\cap D(t+\varepsilon-1)\setminus D(t-1)}\left(\rho-(t-1)\right)^2\,dg\right|\le \varepsilon |\mathcal{R}_X\cap D(t+\varepsilon-1)\setminus D(t-1)|\to 0,
    \\
    &\frac{1}{\varepsilon}\left|\int_{\mathcal{R}_X\cap D(t+\varepsilon)\setminus D(t)}\big(1-(\rho-(t+\varepsilon-1))^2\big)\,dg\right|\le 2|\mathcal{R}_X\cap D(t+\varepsilon)\setminus D(t)|\to 0
    \\
    & \frac{1}{\varepsilon}\left(-\varepsilon\int_{\mathcal{R}_X\cap D(t)\setminus D(t+\varepsilon-1)}\big(2(\rho-(t-1))-\varepsilon\big)\,dg\right)\to -2\int_{\mathcal{R}_X\cap D(t)\setminus D(t-1)}u_t\,dg,
\end{align*}
as $\varepsilon \to 0$. 

\end{proof}

We continue the proof of the proposition. For any regular value $s>0$ of $\rho$, picking a cut-off function $\eta_r$ in Proposition \ref{prop:regular cutoff} with $A\ge 3s+\underline{A}(n)$, $\sigma<3$, and $r\le\overline{r}(A,W,\sigma)$,  we may integrate $\Delta f_0+R=\frac{n}{2}$ on $\mathcal{R}_X\cap D(s)$ and apply the divergence theorem as in the proof of Lemma \ref{CZlem}.
\begin{eqnarray*}
0\le \int_{\mathcal{R}_X\cap D(s)}R\eta_r^2u_t^2\,dg&=&\frac{n}{2}\int_{\mathcal{R}_X\cap D(s)}\eta_r^2u_t^2\,dg-\int_{\mathcal{R}_X\cap D(s)}\Delta f_0\eta_r^2u_t^2\,dg\\
&=&\frac{n}{2}\int_{\mathcal{R}_X\cap D(s)}\eta_r^2u_t^2\,dg+2\int_{D(s)\cap \mathcal{R}_X}u_t^2\eta_r\langle\nabla \eta_r, \nabla f_0\rangle\, dg\\
& &+\int_{D(s)\cap \mathcal{R}_X}\eta_r^2\langle\nabla (u_t^2), \nabla f_0\rangle\, dg-\int_{\{x\in \mathcal{R}_X: \rho(x)=s\}}\eta_r^2u_t^2\frac{\langle\nabla_0 f, \nabla \rho\rangle}{|\nabla\rho|} \,dA\\
&\leq&\frac{n}{2}y(t)+\int_{D(s)\cap \mathcal{R}_X}\eta_r^2\langle\nabla (u_t)^2, \nabla f_0\rangle\, dg+C(A,W,\sigma)r^{3-\sigma}.
\end{eqnarray*}
By letting $r\to 0$ and then $s\to+\infty$, we have
\[
0\le \int_{\mathcal{R}_X}Ru_t^2\,dg\leq \frac{n}{2}y(t)-\frac{t}{2}y'(t),
\]
where we have applied  \eqref{eq:theMW12ineq},  the claim above, and the dominated convergence theorem; \eqref{diffineq} then follows by $\left|\mathcal{R}_X\right|-V(t-1)\leq y(t-1)$ and substituting the above inequality into \eqref{diffeq2}. This completes the proof of the proposition.
\end{proof}

\subsection{Noncompact metric soliton has at least linear volume growth}

\begin{proof}[Proof of Theorem \ref{thm:linear-volume-lower-bound}]
    In view of the arguments in \cite[Theorem 6.1]{MW12} and \cite[Proposition 6]{LW20}, it suffices to verify the following differential inequalities for $t\ge c(n)$
    \begin{equation}\label{1st}
          V(t+1)\leq 2V(t),
    \end{equation}
 \begin{equation}\label{2nd}
    V(t+1)-V(t)\leq C_1\frac{V(t)}{t},
    \end{equation}
   \begin{equation}\label{3rd}
    |D(t+1)\cap\mathcal{R}_X\setminus D(t)|^{\frac{n-2}{n}}\leq C_2e^{-\frac{2W}{n}}\Big[ |D(t)\cap\mathcal{R}_X\setminus D(t-1)|+|D(t+2)\cap\mathcal{R}_X\setminus D(t+1)|+\chi(t+2)-\chi(t-1)\Big],
  \end{equation}
    where $C_1$ and $C_2$ are dimensional constants, c.f. equations (266), (267), and (271) in \cite{LW20}. By the absolute continuity of $V$ and $\chi$, we can mulitply both sides of \eqref{diffeq} by $s^{-n}$ and integrate the equation over $[t,t+1]$ as in \cite{CZ10} to get
    \begin{equation}\label{interm}
    (t+1)^{-n}V(t+1)-t^{-n}V(t)\leq4(t+1)^{-(n+2)}\chi(t+1)\leq 2n(t+1)^{-(n+2)}V(t+1)
   \end{equation}
    for all $t\ge \sqrt{2(n+2)}$; here we also used \eqref{intbddR}. Hence
    \[
    V(t+1)\left(1-\frac{2n}{(t+1)^2}\right)\leq \frac{(t+1)^n}{t^n}V(t)
    \]
    which implies \eqref{1st} for all $t\ge c(n)$. Substitute \eqref{1st} into \eqref{interm} and use $((t+1)^n-t^n)/t^n \leq 2^n/t$ as in \cite{MW12}, we have
    \[
     V(t+1)-V(t)\leq \frac{2^n}{t}V(t)+\frac{4n}{(t+1)^2}V(t).
    \]
    This gives \eqref{2nd}. To prove \eqref{3rd}, we make use of the global Sobolev inequality in Corollary \ref{coro: Sobolev}. For any fixed $t\ge c(n)$ Let 
    \begin{equation*}
    w_t(s)=\left\{
\begin{array}{rl}
0 &\text{  if  } s\le t-1;\\
s-t+1 &\text{  if  } t-1\le s\le t;\\
 1  &\text{  if  } t \le s\le t+1;\\
 t+2-s &\text{  if  } t+1 \le s\le t+2;\\
 0 &\text{  if  } s\ge t+2.
\end{array} 
\right.
\end{equation*}
We also write $w_t(x):=w_t(\rho(x))$. Using \eqref{rhoest}, we may fix $A\ge 3t+\underline{A}(n)$ in Proposition \ref{prop:regular cutoff} to get a cutoff function $\eta_r$ for all $r\le \overline{r}(t,W,\sigma)$, where $0<\sigma<2$ is arbitrarily fixed. We substitute $u:=\eta_r w_t\in C^{0,1}_c(\mathcal{R}_X)$ into the Sobolev inequality of Corollary \ref{coro: Sobolev} to get
\begin{equation}\label{someq}
 \left(\int_{\mathcal{R}_X} |\eta_r w_t|^{\frac{2n}{n-2}}\,dg\right)^{\frac{n-2}{n}}\le C(n)e^{-\frac{2W}{n}}\int_{\mathcal{R}_X}\left(4|\nabla (\eta_r w_t)|^2+ R\eta_r^2 w_t^2\right)\,dg.
 \end{equation}

By Proposition \ref{prop:regular cutoff} and the definition of $w_t$, as $r\to 0$
\[
\int_{\mathcal{R}_X} |(1-\eta_r) w_t|^{\frac{2n}{n-2}}\,dg\leq \int_{\mathcal{R}_X\cap B(x_0,A)\cap\{0\le \eta_r<1\}}\,dg \le C(A,W,\sigma)r^{4-\sigma}\to 0 
\]
and thus by the triangle inequality
\[
 \left(\int_{\mathcal{R}_X} |\eta_r w_t|^{\frac{2n}{n-2}}\,dg\right)^{\frac{n-2}{n}} \to  \left(\int_{\mathcal{R}_X} |w_t|^{\frac{2n}{n-2}}\,dg\right)^{\frac{n-2}{n}}\ge  |D(t+1)\cap\mathcal{R}_X\setminus D(t)|^{\frac{n-2}{n}}.
\]

Using \eqref{Rupb}, we see that,
\begin{align*}
0&\ \le\int_{\mathcal{R}_X}R(1-\eta_r^2) w_t^2\,dg\leq c_n(t+2)^2\int_{\mathcal{R}_X}(1-\eta_r^2) w_t^2
\\
&\ \leq c_n(t+2)^2\int_{\mathcal{R}_X\cap B(x_0,A)\cap\{0\le \eta_r<1\}}\,dg \le C(A,W,\sigma)\cdot (t+2)^2\cdot r^{4-\sigma}\to 0 
\end{align*}
as $r\to 0$, and 
\[
\int_{\mathcal{R}_X}R\eta_r^2 w_t^2\,dg\to \int_{\mathcal{R}_X}R w_t^2\,dg\leq  \int_{D(t+2)\cap\mathcal{R}_X\setminus D(t-1)}R \,dg=\chi(t+2)-\chi(t-1).
\]

On the other hand,
\[
\int_{\mathcal{R}_X}4|\nabla (\eta_r w_t)|^2\,dg=\int_{\mathcal{R}_X}\left(4\eta_r^2|\nabla  w_t|^2+4w_t^2|\nabla  \eta_r|^2+8\eta_rw_t\langle\nabla\eta_r,\nabla w_t \rangle\right)\,dg.
\]
Thanks to the fact $|\nabla \rho|\leq 1$ and Proposition \ref{prop:regular cutoff}, when $r\to 0$,
\begin{eqnarray*}
&&\int_{\mathcal{R}_X}\left(4w_t^2|\nabla  \eta_r|^2+8\eta_rw_t\left|\langle\nabla\eta_r,\nabla w_t \rangle\right|\right)\,dg
\\
&\leq& C\int_{\mathcal{R}_X\cap\{|\nabla\eta_r|\neq 0\}} w_t^2r^{-2}\,dg+C\int_{D(t+2)\cap\mathcal{R}_X} |\nabla\eta_r|\,dg\\
&\leq& C(A,W,\sigma)(r^{2-\sigma}+r^{3-\sigma})\to 0,
\end{eqnarray*}
\[
\int_{\mathcal{R}_X}4(1-\eta_r^2)|\nabla  w_t|^2\leq C\int_{\mathcal{R}_X\cap B(x_0,A)\cap\{0\le \eta_r<1\}}\,dg \le C(A,W,\sigma)r^{4-\sigma}\to 0.
\]
Hence as $r\to 0$,
\begin{eqnarray*}
\int_{\mathcal{R}_X}4|\nabla (\eta_r w_t)|^2\,dg&=&\int_{\mathcal{R}_X}4\eta_r^2|\nabla  w_t|^2+4w_t^2|\nabla  \eta_r|^2+8\eta_rw_t\langle\nabla\eta_r,\nabla w_t \rangle\,dg\\
&\to& \int_{\mathcal{R}_X}4|\nabla w_t|^2\,dg\\
&=&\int_{D(t)\cap\mathcal{R}_X\setminus D(t-1)}4|\nabla w_t|^2\,dg+\int_{D(t+2)\cap\mathcal{R}_X\setminus D(t+1)}4|\nabla w_t|^2\,dg\\
&\leq& 4|D(t)\cap\mathcal{R}_X\setminus D(t-1)|+4|D(t+2)\cap\mathcal{R}_X\setminus D(t+1)|.
\end{eqnarray*}
\eqref{3rd} then follows by letting $r\to 0$ in \eqref{someq}. The linear volume growth is now a consequence of \eqref{1st},\eqref{2nd}, \eqref{3rd}, Proposition \ref{infvol}, and the contradiction arguments in \cite{MW12, LW20}.
\end{proof}

\bigskip

\bibliographystyle{amsalpha}

\newcommand{\alphalchar}[1]{$^{#1}$}
\providecommand{\bysame}{\leavevmode\hbox to3em{\hrulefill}\thinspace}
\providecommand{\MR}{\relax\ifhmode\unskip\space\fi MR }
\providecommand{\MRhref}[2]{%
  \href{http://www.ams.org/mathscinet-getitem?mr=#1}{#2}
}

\bigskip

\bigskip

\noindent Mathematics Institute, Zeeman Building, University of Warwick, Coventry CV4 7AL, UK
\\ E-mail address: \verb"pak-yeung.chan@warwick.ac.uk"
\\

\noindent Department of Mathematics, Rutgers University, Piscataway, NJ 08854, USA
\\ E-mail address: \verb"zilu.ma@rutgers.edu"
\\

\noindent School of Mathematical Sciences, Shanghai Jiao Tong University, Shanghai, China, 200240
\\ E-mail address: \verb"sunzhang91@sjtu.edu.cn"

\end{document}